\documentclass[11pt]{article}
\usepackage{amsmath,amssymb,amsthm,amscd}

\newtheorem{theorem}{Theorem}[section]
\newtheorem{lemma}[theorem]{Lemma}
\newtheorem{proposition}[theorem]{Proposition}
\newtheorem{corollary}[theorem]{Corollary}
\newtheorem{example}[theorem]{Example}

\theoremstyle{plain}

\theoremstyle{definition}
\newtheorem{definition}[theorem]{Definition}

\numberwithin{equation}{section}

\renewcommand{\labelenumi}{\textup{(\theenumi)}}

\title{Cohomology groups, continuous full groups  and\\
 continuous orbit equivalence of topological Markov shifts \\
}
\author{Kengo Matsumoto \\
Department of Mathematics \\
Joetsu University of Education \\
Joetsu, 943-8512, Japan
}

\begin{document}


\maketitle

\date{}

\def\det{{{\operatorname{det}}}}

\begin{abstract}
We will study several subgroups of continuous full groups of one-sided topological Markov shifts
from the view points of cohomology groups of full group actions on the shift spaces.
We also study continuous orbit equivalence and strongly continuous orbit equivalence in terms of these
 subgroups of the continuous full groups and the cohomology groups.
\end{abstract}





\newcommand{\Ker}{\operatorname{Ker}}
\newcommand{\sgn}{\operatorname{sgn}}
\newcommand{\Ad}{\operatorname{Ad}}
\newcommand{\ad}{\operatorname{ad}}
\newcommand{\orb}{\operatorname{orb}}

\def\Re{{\operatorname{Re}}}
\def\det{{{\operatorname{det}}}}
\newcommand{\K}{\mathbb{K}}

\newcommand{\N}{\mathbb{N}}
\newcommand{\C}{\mathbb{C}}
\newcommand{\R}{\mathbb{R}}
\newcommand{\Rp}{{\mathbb{R}}^*_+}
\newcommand{\T}{\mathbb{T}}
\newcommand{\Z}{\mathbb{Z}}
\newcommand{\Zp}{{\mathbb{Z}}_+}
\def\AF{{{\operatorname{AF}}}}

\def\OA{{{\mathcal{O}}_A}}
\def\OB{{{\mathcal{O}}_B}}
\def\OTA{{{\mathcal{O}}_{\tilde{A}}}}
\def\SOA{{{\mathcal{O}}_A}\otimes{\mathcal{K}}}
\def\SOB{{{\mathcal{O}}_B}\otimes{\mathcal{K}}}
\def\SOTA{{{\mathcal{O}}_{\tilde{A}}\otimes{\mathcal{K}}}}
\def\FA{{{\mathcal{F}}_A}}
\def\FB{{{\mathcal{F}}_B}}
\def\DA{{{\mathcal{D}}_A}}
\def\DB{{{\mathcal{D}}_B}}
\def\DZ{{{\mathcal{D}}_Z}}
\def\DTA{{{\mathcal{D}}_{\tilde{A}}}}
\def\Ext{{{\operatorname{Ext}}}}
\def\Max{{{\operatorname{Max}}}}
\def\Per{{{\operatorname{Per}}}}
\def\PerB{{{\operatorname{PerB}}}}
\def\Homeo{{{\operatorname{Homeo}}}}
\def\HA{{{\frak H}_A}}
\def\HB{{{\frak H}_B}}
\def\HSA{{H_{\sigma_A}(X_A)}}
\def\Out{{{\operatorname{Out}}}}
\def\Aut{{{\operatorname{Aut}}}}
\def\Ad{{{\operatorname{Ad}}}}
\def\Inn{{{\operatorname{Inn}}}}
\def\det{{{\operatorname{det}}}}
\def\exp{{{\operatorname{exp}}}}
\def\nep{{{\operatorname{nep}}}}
\def\sgn{{{\operatorname{sign}}}}
\def\cobdy{{{\operatorname{cobdy}}}}
\def\Ker{{{\operatorname{Ker}}}}
\def\ind{{{\operatorname{ind}}}}
\def\id{{{\operatorname{id}}}}
\def\supp{{{\operatorname{supp}}}}
\def\co{{{\operatorname{co}}}}
\def\scoe{{{\operatorname{scoe}}}}
\def\ucoe{{{\operatorname{ucoe}}}}
\def\coe{{{\operatorname{coe}}}}
\def\HGFE{{{\operatorname{H_GFE}}}}
\def\HG1FE{{{\operatorname{H_G^1FE}}}}

\def\S{\mathcal{S}}

\def\coe{{{\operatorname{coe}}}}
\def\scoe{{{\operatorname{scoe}}}}
\def\uoe{{{\operatorname{uoe}}}}
\def\ucoe{{{\operatorname{ucoe}}}}
\def\event{{{\operatorname{event}}}}

\section{Introduction and Preliminary}
For an irreducible non-permutation matrix $A =[A(i,j)]_{i,j=1}^N$ with entries in $\{0,1\},$  a right one-sided topological Markov shifts denoted by $(X_A, \sigma_A)$
consists of the set $X_A$ of right one-sided sequences
$(x_n)_{n\in \N}$ of $x_n \in \{1,2,\dots,N\}$ 
such that $A(x_n, x_{n+1}) =1, n\in \N$, and the one-sided shift
$\sigma_A((x_n)_{n\in \N}) = (x_{n+1})_{n\in \N}$. 
The set $X_A$ is endowed with its infinite product topology to make it 
a compact Hausdorff space, so that $\sigma_A:X_A\longrightarrow X_A$
is a continuous surjection. 
For general theory of symbolic dynamical systems, see the text books \cite{Ki} and \cite{LM}. 
Throughout the paper, matrices $A$ and $B$ are assumed to be square, irreducible
non-permutation matrices with entries in $\{0,1\}.$
Let us denote by $C(X_A,\Z)$ the abelian group of continuous functions from $X_A$ to $\Z$.
The symbol $\Zp$ denotes the set of nonnegative integers.

In \cite{MaPacific}, the author introduced an equivalence relation 
in one-sided topological Markov shifts, called continuous orbit equivalence.
The notion has been closely related to the classification of 
not only a certain class of $C^*$-algebras called the Cuntz--Krieger algebras $\OA$, 
but also a certain class class of countable infinite discrete non-amenable groups $\Gamma_A$
called the continuous full groups associated to $(X_A,\sigma_A)$. 
One-sided topological Markov shifts
$(X_A, \sigma_A)$ and $(X_B,\sigma_B)$ 
are said to be continuously orbit equivalent written
$(X_A, \sigma_A)\underset{\coe}{\sim}(X_B,\sigma_B)$ 
if there exist a homeomorphism
$h: X_A \rightarrow X_B$  and continuous functions 
$k_1,l_1: X_A\rightarrow \Zp,
 k_2,l_2: X_B\rightarrow \Zp 
$
such that
\begin{align}
\sigma_B^{k_1(x)} (h(\sigma_A(x))) 
& = \sigma_B^{l_1(x)}(h(x)) \quad \text{ for } \quad 
x \in X_A,  \label{eq:orbiteq1x} \\
\sigma_A^{k_2(w)} (h^{-1}(\sigma_B(w))) 
& = \sigma_A^{l_2(w)}(h^{-1}(w)) \quad \text{ for } \quad 
w \in X_B. \label{eq:orbiteq2y}
\end{align}
Two continuous functions $c_1: X_A\longrightarrow \Z$ and 
 $c_2: X_B\longrightarrow \Z$ 
defined by
$c_1 :=  l_1 - k_1$
and
$c_2 :=  l_2 - k_2$
are called the cocycle functions for 
the continuous orbit equivalent map
$h:X_A\longrightarrow X_B.$
H. Matui and the author proved that
$(X_A, \sigma_A)\underset{\coe}{\sim}(X_B,\sigma_B)$ if and only if 
the Cuntz--Krieger algebra $\OA$ is isomorphic to 
the Cuntz--Krieger algebra $\OB$ and $\sgn(\det(\id -A)) =\sgn(\det(\id -B))$ 
(\cite{MMKyoto}, cf. \cite{MaPAMS2013}), where $\sgn$ is the plus or minus signature.
The proof of the only if part given in \cite{MMKyoto}
is based on groupoid technique.
In \cite{MMETDS}, it was shown that for a homeomorphism
$h:X_A\longrightarrow X_B$ that gives rise to
$(X_A, \sigma_A)\underset{\coe}{\sim}(X_B,\sigma_B)$,
the map $\varPsi_h: C(X_B,\Z)\longrightarrow C(X_A,\Z)$
defined by
\begin{equation}\label{eq:Psihf}
\varPsi_h(f)(z) = \sum_{i=0}^{l_1(z)}f(\sigma_B^i(h(z))) -
 \sum_{j=0}^{k_1(z)}f(\sigma_B^j(h(\sigma_A(z))),
\qquad f \in C(X_B,\Z), \, z \in X_A
\end{equation} 
induces an ordered isomorphism between the ordered cohomology groups
$(H^A, H^A_+)$ and $(H^B, H^B_+)$.
It implies the determinant condition
$\sgn(\det(\id -A)) =\sgn(\det(\id -B))$
 by the results of \cite{PS} and \cite{BH}.
Studies for more general matrices related to graph algebras have been seen in
several papers (cf. \cite{CEOR}, \cite{CRS}, etc.)

In classification under orbit equivalences  
of topological dynamical systems, 
subgroups called full groups of homeomorphisms preserving orbit structure of the dynamics
have played a crucial r\^ole as seen in 
\cite{GPS2}, \cite{MatuiPLMS}, \cite{Tom2}, etc.
For the class of one-sided topological Markov shifts, 
 the author in \cite{MaPacific}
 introduced the continuous full group $\Gamma_A$
that 
has been written $[\sigma_A]_c$ in the earlier papers
(\cite{MaPacific}, \cite{MaDCDS}).
The group
$\Gamma_A$ is a subgroup of the group
$\Homeo(X_A)$ of
homeomorphisms on $X_A$
consisting of 
$\tau\in \Homeo(X_A)$  
such that there exist continuous functions 
$k_\tau, l_\tau : X_A \rightarrow \Zp$ 
satisfying
\begin{equation}
\sigma_A^{k_\tau(x)}(\tau(x) )=\sigma_A^{l_\tau(x)}(x)
\quad\text{ for all } x \in X_A. \label{eq:tau}
\end{equation}
For each $\tau \in \Gamma_A$, 
the function $d_\tau: X_A \longrightarrow \Z$ is defined by
$d_\tau = l_\tau - k_\tau$. 
It does not depend on the choice of $l_\tau$ and $k_\tau$ (\cite[Lemma 7.6]{MMGGD}).
The group $\Gamma_A$ is a countably infinite discrete non-amenable group
(\cite{MaDCDS}, see \cite{MatuiPLMS} for more general results). 
H. Matui has shown that 
the group $\Gamma_A$, written $[G_A]$ in \cite{MatuiPLMS}, \cite{MatuiCrelle}, 
 itself has several remarkable properties.
He proved that
$\Gamma_A$
is not only finitely generated but also finitely presented
for every $A$ as a group.
He also found simplicity condition on th groups $\Gamma_A$
and pointed out an interesting relationship to Thompson groups
and Higman--Thompson groups (see \cite{MatuiCrelle},
cf. \cite{MMGGD}).

For the Cuntz--Krieger algebra $\OA$ with its canonical generating partial isometries
$S_1,\dots,S_N$ satisfying $\sum_{i=1}^N S_j S_j^* =1$
and $S_i^* S_i =\sum_{i=1}^N A(i,j)S_j S_j^* =1, \, i=1,\dots,N$,
let us denote by $\DA$ the commutative $C^*$-subalgebra of $\OA$
generated by the projections of the form 
$S_\mu S_\mu^*$ for all admissible words $\mu=(\mu_1,\dots, \mu_n)$ of $X_A$, 
where 
$S_\mu$ denotes $S_{\mu(1)}\cdots S_{\mu(n)}$.
The gauge action $\rho^A: t \in \R/\Z =\T \longrightarrow \rho^A_t\in \Aut(\OA)$ 
is defined by $\rho^A_t(S_j) = e^{2 \pi \sqrt{-1}t} S_j, \, j=1,\dots,N$.
A function $f \in C(X_A,\Z)$ is regarded as an element of $\DA$
and then written as 
$f = \sum_{i=}^n f_i S_{\mu(i)} S_{\mu(i)}^*$
for some integers $f_i\in \Z$ 
and admissible words $\mu(i), i=1,\dots,n$ of $X_A$
satisfying $\sum_{i=1}^n S_{\mu(i)}S_{\mu(i)}^*=1$.
Then an element $e^{2 \pi \sqrt{-1}t f}$ in $\DA$
for $t \in \T$ is regarded as the unitary
$\sum_{i=1}^n e^{2 \pi \sqrt{-1} f_i t}S_{\mu(i)}S_{\mu(i)}^*$,
and the generalized gauge action 
$\rho^{A,f}: t \in \T \longrightarrow \rho^{A,f}_t\in \Aut(\OA)$ 
is defined by $\rho^{A,f}_t(S_j) = e^{2 \pi \sqrt{-1}t f} S_j, \, j=1,\dots,N$.

Let $A$ and $B$  be irreducible, non permutation matrices over $\{0,1\}$.
In \cite[Corollary 1.2]{MaIsrael2015}, 
the following characterization of continuous orbit equivalence was proved. 
\begin{theorem}[{\cite[Corollary 1.2]{MaIsrael2015}}]
The following three conditions are equivalent:
\begin{enumerate}
\renewcommand{\theenumi}{\roman{enumi}}
\renewcommand{\labelenumi}{\textup{(\theenumi)}}
\item $(X_A,\sigma_A)$ and $(X_B,\sigma_B)$ are continuously orbit equivalent. 
\item There exists an isomorphism $\Phi:\OA\longrightarrow \OB$ 
such that 
$\Phi(\DA) = \DB$.
\item  
The continuous full groups $\Gamma_A$ and $\Gamma_B$ are isomorphic 
as groups.
\end{enumerate}
\end{theorem}
For a homeomorphism $h:X_A\longrightarrow X_B$ giving rise to
$(X_A, \sigma_A)\underset{\coe}{\sim}(X_B,\sigma_B)$,
the isomorphism $\xi_h: \Gamma_A \longrightarrow \Gamma_B$ in (iii)
is given by $\xi_h(\tau) = h\circ \tau \circ h^{-1}$ for $\tau \in \Gamma_A$. 

After completing the classification results \cite[Theorem 3.6]{MMKyoto}
of continuous orbit equivalence of one-sided irreducible topological Markov shifts,
the author introduced in \cite{MaJOT}, \cite{MaPAMS2017} 
 three more equivalence relations slightly stronger 
than continuous orbit equivalence.
The first one is called strongly continuous orbit equivalence,
the second one is called uniformly continuous orbit equivalence,
and the third one is eventual conjugacy. 
Two one-sided topological Markov shifts
$(X_A,\sigma_A)$ and $(X_B, \sigma_B)$ 
are said to be {\it strongly continuous orbit equivalent}
 written
$(X_A, \sigma_A)\underset{\scoe}{\sim}(X_B,\sigma_B)$
if they are continuously  orbit equivalent
and one may take its cocycle functions of the form
$c_1 = 1 - b_1 + b_1\circ \sigma_A$
and
$c_2 = 1 - b_2 + b_2\circ \sigma_B$
for some continuous functions
$b_1 \in C(X_A,\Z)$ and $b_2 \in C(X_B,\Z)$.
If in particular, one may take $b_1 \equiv 0$ and $b_2\equiv 0$,
$(X_A,\sigma_A)$ and $(X_B, \sigma_B)$ 
are said to be {\it eventually conjugate}.
The one-sided topological Markov shifts for the matrices 
$\begin{bmatrix}
1 & 1 \\
1 & 1 
\end{bmatrix}
$
and
$\begin{bmatrix}
1 & 1 \\
1 & 0 
\end{bmatrix}
$
are continuous orbit equivalent, but not strongly continuous orbit equivalent
(\cite[Section 3]{MaPAMS2017}).
As in  \cite[Section 3]{MaPAMS2017},
there are two irreducible non-permutation matrices such that their 
one-sided topological Markov shifts are strongly continuous orbit equivalent,
but not eventually conjugate. 
Let $\Gamma_A^{\AF}$ be the subgroup of $\Gamma_A$ consisting of 
$\tau \in \Gamma_A$ such that $k_\tau(x) = l_\tau(x), x \in X_A$
in \eqref{eq:tau}.
The subgroup $\Gamma_A^{\AF}$ is called the AF full group of $(X_A,\sigma_A)$.
Two one-sided topological Markov shifts
$(X_A,\sigma_A)$ and $(X_B, \sigma_B)$ 
are said to be {\it uniformly continuous orbit equivalent}
written
$(X_A, \sigma_A)\underset{\ucoe}{\sim}(X_B,\sigma_B)$
if there exists a homeomorphism that gives rise to
a continuous orbit equivalence between them and
 for any $\tau_1 \in \Gamma_A^{\AF},\tau_2 \in \Gamma_B^{\AF}$
there exist natural numbers $K_{\tau_1}, 
K_{\tau_2}$
such that  
\begin{align*}
\sigma_B^{K_{\tau_1}}(h(\tau_1(x))) & = \sigma_B^{K_{\tau_1}}(h(x)), \qquad x \in X_A,\\
\sigma_A^{K_{\tau_2}}(h^{-1}(\tau_2(w))) & = \sigma_A^{K_{\tau_2}}(h^{-1}(w)), \qquad w \in X_B.
\end{align*}
In \cite[Theorem 1.5]{MaPAMS2017},
it was proved that 
eventual conjugacy is equivalent to uniformly continuous orbit equivalence.
Consequently one knows the following implications:
\begin{align*}
& \text{eventually conjugate} \\ 
& \qquad \qquad \Vert\\
& \qquad \quad \underset{\ucoe}{\sim} \qquad
\Longrightarrow  \qquad
\underset{\scoe}{\sim} \qquad
\Longrightarrow  \qquad
\underset{\coe}{\sim} 
\end{align*}
The following characterization of uniformly continuous orbit equivalence
was shown.

\begin{theorem}[{\cite[Theorem 1.5]{MaPAMS2017}}] \label{thm:uniformlycoe}
The following four conditions are equivalent:
\begin{enumerate}
\renewcommand{\theenumi}{\roman{enumi}}
\renewcommand{\labelenumi}{\textup{(\theenumi)}}
\item $(X_A,\sigma_A)$ and $(X_B,\sigma_B)$ are uniformly  continuous orbit equivalent. 
\item There exists an isomorphism $\Phi:\OA\longrightarrow \OB$ 
 such that 
$\Phi(\DA) = \DB$ and
\begin{equation*}
\Phi \circ \rho^A_t =  \rho^B_t \circ \Phi, \qquad t \in \T.
\end{equation*}
\item
There exists an isomorphism $\xi: \Gamma_A \longrightarrow \Gamma_B$ 
of groups  such that
$\xi(\Gamma_A^{\AF}) =\Gamma_B^{\AF}.$ 
\item
There exists a homeomorphism 
$h: X_A \longrightarrow X_B$ 
 such that 
\begin{equation*}
\xi_h(\Gamma_A) =\Gamma_B \quad \text{ and } \quad
\xi_h(\Gamma_A^{\AF}) =\Gamma_B^{\AF}.
\end{equation*}  
\end{enumerate}
\end{theorem}
This theorem, in particular, the equivalence between (i) and (iii)
shows that the uniformly continuous orbit equivalence class of $(X_A,\sigma_A)$
is completely  determined by the pair
$(\Gamma_A, \Gamma_A^{\AF})$
of the group $\Gamma_A$ and its subgroup $\Gamma_A^{\AF}$.
K. A. Brix and T. M. Carlsen in \cite{BC} recently studied one-sided topological conjugacy
of one-sided topological Markov shifts from the view points of groupoids, 
and showed that 
 uniformly continuous orbit equivalence is strictly weaker than one-sided topological conjugacy.

In the first half of the paper, 
we will first study the first cohomology group
$H^1(\Gamma_A,X_A;\Z)$ for the action of the full group $\Gamma_A$ on $X_A$,
and introduce a cocycle subgroup $\Gamma_{A,f}$ of $\Gamma_A$ for $f \in C(X_A,\Z)$
and a coboundary subgroup 
$\Gamma_A^b$ for $b \in C(X_A,\Z)$ in the following way.
Let $f \in C(X_A,\Z)$. 
For $(x, \tau) \in X_A\times \Gamma_A$, 
we set
\begin{equation} \label{eq:rhof2}
\rho^f(x,\tau)  
= \sum_{i=0}^{l_\tau(x)}f(\sigma_A^i(x)) -\sum_{j=0}^{k_\tau(x)}f(\sigma_A^j(\tau(x))).
\end{equation}
it is easy to see that $\rho^f(x,\tau)$ does not depend on the choice of $\l_\tau$ and $k_\tau$ 
as long as they satisfy \eqref{eq:tau}.
\begin{definition} Let $f, b $ be integer valued continuous functions on $X_A$.
\begin{enumerate}
\renewcommand{\theenumi}{\roman{enumi}}
\renewcommand{\labelenumi}{\textup{(\theenumi)}}
\item
The cocycle subgroup $\Gamma_{A,f}$ of $\Gamma_A$ is defined by
\begin{equation}\label{eq:cocyclesub}  
\Gamma_{A,f} : = \{ \tau \in \Gamma_A 
\mid \rho^f(x,\tau) =0 \text{ for all } x \in X_A \}.
\end{equation}
\item
The coboundary subgroup $\Gamma_A^b$ of $\Gamma_A$ is defined by
\begin{equation}\label{eq:coboundarysub}
\Gamma_A^b : = \{ \tau \in \Gamma_A 
\mid d_\tau(x ) = b(x) - b(\tau(x)) \text{ for all } x \in X_A \},
\end{equation} 
where $d_\tau = l_\tau - k_\tau$.
\end{enumerate}
\end{definition}
It will be shown that they are actually subgroups of $\Gamma_A$
(Proposition \ref{prop:cocyclesubgroup}, Lemma \ref{lem:1b}).
For the constant functions $f\equiv 1$ and $b\equiv 0$,
the subgroups satisfy 
$\Gamma_{A,1} = \Gamma_A^0 = \Gamma_A^\AF$.

In the second half of the paper, 
we will establish  the following three theorems.
The first  theorem below describes behavior of cocycle subgroups and cohomology groups
as well as gauge actions under  continuous orbit equivalence.

\begin{theorem}\label{thm:main1}
Let $A$ and $B$ be irreducible, non permutation matrices with entries in $\{0,1\}$. 
Suppose that $(X_A,\sigma_A)$ and $(X_B,\sigma_B)$ are continuously orbit equivalent
via homeomorphism $h: X_A\longrightarrow X_B$.
Let $\varPsi_h: C(X_B,\Z)\longrightarrow C(X_A,\Z)$
be the homomorphism defined by \eqref{eq:Psihf}.
Then the following assertions hold.
\begin{enumerate}
\renewcommand{\theenumi}{\roman{enumi}}
\renewcommand{\labelenumi}{\textup{(\theenumi)}}
\item There exists an isomorphism $\Phi:\OA\longrightarrow \OB$ 
such that 
$\Phi(\DA) = \DB$ and
\begin{equation*}
\Phi \circ \rho^{A, \varPsi_h(f)}_t = \rho^{B,f}_t \circ \Phi, \qquad f \in C(X_B,\Z), \, \, t \in \T.
\end{equation*}
\item
There exists an isomorphism $\xi: \Gamma_A \longrightarrow \Gamma_B$ 
of groups  such that 
\begin{equation*}
\xi(\Gamma_{A, \varPsi_h(f)}) =\Gamma_{B,f}, \qquad f \in C(X_B,\Z).
\end{equation*}  
\item
There exists  an isomorphism
$
\varPhi_h: H^1(\Gamma_A, X_A;\Z) \longrightarrow H^1(\Gamma_B, X_B;\Z) 
$
of the cohomology groups such that
\begin{equation*}
\varPhi_h([\rho^{\varPsi_h(f)}]) = [\rho^f], \qquad f \in C(X_B, \Z)
\end{equation*}
where 
$\rho^f(y,\varphi) =\sum_{i=0}^{l_{\varphi}(y)}f(\sigma_B^i(y)) -
\sum_{j=0}^{k_{\varphi}(y)}f(\sigma_B^j(\varphi(y)))
$
for $(y,\varphi) \in X_B\times\Gamma_B$.
\end{enumerate}
\end{theorem}
 The assertion  (i)  above has been already proved in 
\cite[Theorem 3.2]{MaMZ}.
Hence we just need to show the assertions  (ii) and (iii).

If we restrict our interest to strongly continuous orbit equivalence, 
we have the following second theorem, which is a characterization of
strongly continuous orbit equivalence in terms of  coboundary subgroups and cohomology groups. 
In the theorem, we further assume that the matrices are primitive,
that is, irreducible and aperiodic.
\begin{theorem}
\label{thm:main2}
Let $A$ and $B$ be primitive matrices with entries in $\{0,1\}$. 
Then the following five assertions are equivalent.
\begin{enumerate}
\renewcommand{\theenumi}{\roman{enumi}}
\renewcommand{\labelenumi}{\textup{(\theenumi)}}
\item $(X_A,\sigma_A)$ and $(X_B,\sigma_B)$ are strongly continuous orbit equivalent. 
\item There exist an isomorphism $\Phi:\OA\longrightarrow \OB$ 
and a unitary one-cocycle $v_t \in \OB, t \in \T$ for gauge action
$\rho^B$ on $\OB$ such that 
$\Phi(\DA) = \DB$ and
\begin{equation*}
\Phi \circ \rho^A_t = \Ad(v_t)\circ \rho^B_t \circ \Phi, \qquad t \in \T.
\end{equation*}
\item
There exist an isomorphism $\xi: \Gamma_A \longrightarrow \Gamma_B$ 
of groups and continuous functions
$b_1 \in C(X_A,\Z), b_2 \in C(X_B,\Z)$ such that 
\begin{equation*}
\xi(\Gamma_A^{\AF}) =\Gamma_B^{b_2}, \qquad
\xi^{-1}(\Gamma_B^{\AF}) =\Gamma_A^{b_1}.
\end{equation*}  
\item
There exist a homeomorphism 
$h: X_A \longrightarrow X_B$ 
and continuous functions
$b_1 \in C(X_A,\Z), b_2 \in C(X_B,\Z)$ such that 
$\xi_h(\Gamma_A) =\Gamma_B$ and
\begin{gather*}
\xi_h(\Gamma_A^f) =\Gamma_B^{f\circ h^{-1} + b_2}, \qquad f \in C(X_A,\Z), \\
\xi_{h^{-1}}(\Gamma_B^g) =\Gamma_A^{g\circ h + b_1}, \qquad g \in C(X_B,\Z). 
\end{gather*}  
\item
There exists a homeomorphism 
$h: X_A \longrightarrow X_B$ 
such that 
$\xi_h(\Gamma_A) =\Gamma_B$ and
the induced isomorphism
$$
\varPhi_h: H^1(\Gamma_A, X_A;\Z) \longrightarrow H^1(\Gamma_B, X_B;\Z) 
$$
of the cohomology groups satisfies
$\varPhi_h([d^A]) = [d^B],$
where
$d^A(x,\tau) = l_{\tau}(x) -k_{\tau}(x)$ for $(x,\tau) \in X_A\times \Gamma_A$,
$d^B(y,\varphi) = l_{\varphi}(y) -k_{\varphi}(y)$ for
 $(y,\varphi) \in X_B\times \Gamma_B$.
\end{enumerate}
\end{theorem}

 The equivalence between the first two assertions has been already proved in \cite[Theorem 6.7]{MaJOT}.
The conditions  
$\xi(\Gamma_A^{\AF}) =\Gamma_B^{b_2},
\xi^{-1}(\Gamma_B^{\AF}) =\Gamma_A^{b_1}
$
for $b_1\equiv 0, b_2\equiv 0$ in the above assertion (iii) 
exactly reduce to the condition
$\xi(\Gamma_A^\AF)=\Gamma_B^\AF$. 
The unitary one-cocycle $v_t$ appeared in the above assertion (ii)
is the unitary defined by $e^{2\pi\sqrt{-1}b_2 t}$ for the continuous function $b_2 \in C(X_B,\Z)$ 
in (iii) as in the proof of \cite[Proposition 6.5]{MaJOT}.
Hence the above first four conditions (i), (ii) , (iii) and (iv)
are counterparts of Theorem \ref{thm:uniformlycoe}
that states equivalent conditions of uniformly continuous orbit equivalence.

Let $h:X_A\longrightarrow X_B$ 
be a homeomorphism that gives rise to a continuous orbit equivalence
between $(X_A,\sigma_A)$ and $(X_B,\sigma_B)$ as in 
\eqref{eq:orbiteq1x} and \eqref{eq:orbiteq2y}.
Recall that the cocycle functions are defined by 
$ c_1 = l_1 - k_1$ and $c_2 = l_2 - k_2$.
Recall also that one-sided topological Markov shifts 
$(X_A,\sigma_A)$ and $(X_B,\sigma_B)$ are strongly continuous orbit equivalent
if there exists $b_1\in C(X_A, \Z)$ such that
$c_1 = 1 + b_1 - b_1\circ\sigma_A$.
If we may take the above function $b_1$ as 
$b_1 = -d_\tau$ where $d_\tau = l_\tau - k_\tau$ 
for some $\tau \in \Gamma_A$, 
then $(X_A,\sigma_A)$ and $(X_B,\sigma_B)$ are said to be 
$\Gamma$-{\it strongly continuous orbit equivalent}.
We then finally prove that 
 the one-sided topological Markov shifts 
$(X_A, \sigma_A)$ and $(X_B, \sigma_B)$ are 
$\Gamma$-strongly continuous orbit equivalent
if and only if they are eventually conjugate (Theorem \ref{thm:main6.6}).
Since eventual conjugacy is equivalent to uniformly continuous orbit equivalence
(\cite[Theorem 1.5]{MaPAMS2017}),
we obtain the following result as the third theorem in the present paper.
\begin{theorem}[{Corollary \ref{cor:main6.7}}]
Let $A, B$ be irreducible, non-permutation matrices with entries in 
$\{0,1\}$.
Then the following three conditions are equivalent:
\begin{enumerate}
\renewcommand{\theenumi}{\roman{enumi}}
\renewcommand{\labelenumi}{\textup{(\theenumi)}}
\item $(X_A,\sigma_A)$ and $(X_B,\sigma_B)$ are uniformly continuous orbit equivalent. 
\item $(X_A,\sigma_A)$ and $(X_B,\sigma_B)$ are eventually conjugate. 
\item $(X_A, \sigma_A)$ and $(X_B, \sigma_B)$ are $\Gamma$-strongly continuous orbit equivalent.
\end{enumerate}
\end{theorem}

\medskip

The paper is organized in the following way.
In Section 2, we study the first cohomology group $H^1(\Gamma_A, X_A;\Z)$.
In Section 3, it is shown that continuous orbit equivalence preserves 
the first cohomology group $H^1(\Gamma_A, X_A;\Z)$.
In Section 4, the proof of Theorem \ref{thm:main1} is given.
In Section 5, a relationship between strongly continuous orbit equivalence and coboundary subgroups
is studied, and the proof of Theorem \ref{thm:main2} is completed.
In Section 6, it is proved that $\Gamma$-strongly continuous orbit equivalence is equivalent to eventual conjugacy.  

\medskip 

Let us provide some notation.
Denote by $B_k(X_A)$ the set of admissible words of $X_A$ with length $k$.
For $\mu = (\mu_1,\dots, \mu_k) \in B_k(X_A)$,
let us denote by $U_\mu$ the cylinder set 
$
U_\mu =\{ (x_i)_{i\in \N} \in X_A \mid x_1 = \mu_1,\dots, x_k = \mu_k \}.
$ 
For $f \in C(X_A, \Z)$ and $m \in \N$, we set 
$
f^m(x) =\sum_{i=0}^{m-1}f(\sigma_A^i(x)), \, x \in X_A.
$
The formula
\begin{equation}\label{eq:fmn}
f^{n+m}(x) = f^m(x) + f^n(\sigma_A^m(x)), \qquad n,  m \in \Zp,\,\, x \in X_A 
\end{equation}
is straightforward to verify and useful in our further discussions.


\section{The first cohomology group}
Recall that a homeomorphism $\tau$ on $X_A$ belongs to
the continuous full group $\Gamma_A$ of $(X_A,\sigma_A)$ 
if and only if  
there exist continuous functions
$k_\tau, l_\tau:X_A \longrightarrow \Zp$ 
satisfying \eqref{eq:tau}. 
The function $d_\tau: X_A \longrightarrow \Z$ defined by
$d_\tau(x) = l_\tau(x) - k_\tau(x), \, x \in X_A$ 
does not depend on the choice of 
$k_\tau, l_\tau$ satisfying \eqref{eq:tau} (\cite[Lemma 7.6]{MMGGD}). 
The following lemma is straightforward to verify
 and useful in our further discussion.
\begin{lemma}[{\cite[Lemma 7.7]{MMGGD}}] \label{lem:2.1}
For $\tau_1,\tau_2 \in \Gamma_A$, the  formulas   
\begin{equation*}
l_{\tau_2\circ \tau_1} = l_{\tau_1}  + l_{\tau_2} \circ \tau_1, \quad
k_{\tau_2\circ \tau_1} =k_{\tau_1} + k_{\tau_2} \circ \tau_1 \quad 
\text{ and hence } \quad
d_{\tau_2\circ \tau_1} =d_{\tau_1} + d_{\tau_2} \circ \tau_1
\end{equation*}
hold.
\end{lemma} 
A continuous function
$\rho: X_A \times \Gamma_A \longrightarrow \Z$
is called a {\it one-cocycle}\/ 
if it satisfies 
\begin{equation}
\rho(x,\tau_2 \circ \tau_1) 
=
\rho(x, \tau_1) + \rho(\tau_1(x), \tau_2), \qquad x \in X_A, \, \, \tau_1, \tau_2 \in \Gamma_A.
\label{eq:cocycle}
\end{equation}
The identity \eqref{eq:cocycle} implies that
$\rho(x,\id) =0$ for all $x \in X_A$.
For $g \in C(X_A,\Z)$, 
define a {\it coboundary}\/
$\delta_g :X_A \times \Gamma_A \longrightarrow \Z$
by setting
$$
\delta_g(x,\tau) = g(x) - g(\tau(x)), \qquad (x, \tau) \in X_A\times\Gamma_A.
$$
As 
\begin{equation*}
\delta_g(x,\tau_2\circ\tau_1) 
= g(x) - g(\tau_1(x)) + g(\tau_1(x))  - g(\tau_2(\tau_1(x))) 
 = \delta_g(x,\tau_1) + \delta_g(\tau_1(x),\tau_2)
\end{equation*}
for $\tau_1, \tau_2 \in \Gamma_A$, a coboundary is a one-cocycle.
Denote by
$Z^1(\Gamma_A,X_A; \Z) $ and 
$B^1(\Gamma_A,X_A; \Z) $ 
the set of one-cocycles
and that of coboundaries, respectively.
Then 
$Z^1(\Gamma_A,X_A; \Z)$ 
becomes an abelian group with natural addition 
such that 
$B^1(\Gamma_A,X_A; \Z)$ 
is contained as a subgroup of 
$Z^1(\Gamma_A,X_A; \Z)$.
\begin{definition}
The {\it first cohomology group}\/
$H^1(\Gamma_A,X_A; \Z)$ is defioned to be 
the quotient group: 
\begin{equation*}
H^1(\Gamma_A,X_A; \Z)
=Z^1(\Gamma_A,X_A; \Z)/ B^1(\Gamma_A,X_A; \Z).
\end{equation*}
\end{definition}
We set
$Z^1_+(\Gamma_A,X_A; \Z) 
=\{
\rho \in Z^1(\Gamma_A,X_A; \Z) \mid \rho(x,\tau) \ge 0 \, \forall x \in X_A,\tau \in \Gamma_A\}.
$
We define a positive cone $H^1_+(\Gamma_A,X_A; \Z)$
by 
$$
H^1_+(\Gamma_A,X_A; \Z) =
\{ [\rho] \in H^1(\Gamma_A,X_A; \Z) \mid \rho \in Z^1_+(\Gamma_A,X_A; \Z) \}.
$$
We are always assuming that the matrix $A$ is irreducible and not a permutation. 
\begin{lemma}
$
H^1_+(\Gamma_A,X_A; \Z) \cap (-H^1_+(\Gamma_A,X_A; \Z)) = \{ 0\}.
$
\end{lemma}
\begin{proof}
Take $[\rho] \in H^1_+(\Gamma_A,X_A; \Z) \cap (-H^1_+(\Gamma_A,X_A; \Z))$
and
$\rho_1,  \rho_2 \in Z^1_+(\Gamma_A,X_A; \Z)$ such that
$[\rho] = [\rho_1] = [-\rho_2]$.
There exist $g_1, g_2 \in C(X_A, \Z)$ such that 
\begin{equation*}
(\rho - \rho_1)(x,\tau) = g_1(x) - g_1(\tau(x)),
\quad
(\rho + \rho_2)(x,\tau) = g_2(x) - g_2(\tau(x))
\quad \text{ for all } (x, \tau) \in X_A\times\Gamma_A.
\end{equation*}
Hence we have
$
\rho_1(x,\tau) + \rho_2(x,\tau) = (g_2-g_1)(x) - (g_2-g_1)(\tau(x)).
$
Put $g_3 = g_2 - g_1$.
Since
$\rho_1(x,\tau), \rho_2(x,\tau) \ge 0$ for all
$ x \in X_A, \tau \in \Gamma_A,$
we have 
\begin{equation}
g_3(x) - g_3(\tau(x)) \ge 0
\qquad \text{ for all }
\quad
(x, \tau) \in X_A\times\Gamma_A. \label{eq:g31}
\end{equation}
Suppose that $g_3$ is not constant.
Take $z, w \in X_A$ such that $g_3(z) \ne g_3(w)$.
We may assume that $g_3(z) < g_3(w)$.
As $g_3$ is continuous and the matrix $A$
is irreducible,
one may find $\tau_0\in \Gamma_A$ such that 
$g_3(w) = g_3(\tau_0(z))$, 
a contradiction to \eqref{eq:g31}.
Therefore $g_3$ is a constant which we denote by $c$.
Since 
$
\rho_1(x,\tau) + \rho_2(x,\tau) = g_3(x) - g_3(\tau(x)) =c - c =0,
$
and 
$\rho_i(x,\tau) \ge 0, i=1,2$
for all $x \in X_A,\tau \in \Gamma_A$,
we conclude that $\rho_1 \equiv \rho_2 \equiv 0$ and hence 
$[\rho] = 0  \in H^1(\Gamma_A,X_A; \Z).$
\end{proof}
Therefore we have
\begin{proposition}
The first cohomology group 
$(H^1(\Gamma_A,X_A;\Z), H^1_+(\Gamma_A,X_A; \Z))$
is an ordered group.
\end{proposition}

For $f \in C(X_A,\Z)$,
recall that the function
$\rho^f:X_A\times \Gamma_A \longrightarrow \Z$
is defined by \eqref{eq:rhof2}. 
Notice that   
\begin{equation} \label{eq:rhof21}
\rho^f(x,\tau)  
= f^{l_\tau(x)}(x) - f^{k_\tau(x)}(\tau(x)), \qquad
(x,\tau) \in X_A\times \Gamma_A. 
\end{equation}
For $f \equiv 1_{X_A}$,
we set
$ d^A(x,\tau) :=\rho^1(x,\tau)= d_\tau(x) \,(=l_{\tau}(x) - k_{\tau}(x)), \, (x,\tau) \in X_A\times \Gamma_A$.
\begin{lemma}\label{lem:rhoftau21}
For $f \in C(X_A,\Z)$, we have
\begin{equation*}
\rho^f(x,\tau_2 \circ \tau_1) 
=
\rho^f(x, \tau_1) + \rho^f(\tau_1(x), \tau_2), \qquad x \in X_A, \tau_1, \tau_2 \in \Gamma_A.
\end{equation*}
Hence $\rho^f \in Z^1(\Gamma_A, X_A;\Z)$ and 
$[\rho^f] \in H^1(\Gamma_A, X_A;\Z).$ 
\end{lemma}
\begin{proof}
For $ x \in X_A, \tau_1, \tau_2 \in \Gamma_A,$ 
we have by Lemma \ref{lem:2.1} and \eqref{eq:fmn}
\begin{align*}
\rho^f(x,\tau_2 \circ \tau_1) 
= &f^{l_{\tau_1}(x) +l_{\tau_2}(\tau_1(x))} (x) 
   - f^{k_{\tau_1}(x) +k_{\tau_2}(\tau_1(x))}{(\tau_2 \circ\tau_1}(x)) \\ 
=& f^{l_{\tau_1}(x)}(x)  + f^{l_{\tau_2}(\tau_1(x))} (\sigma_A^{l_{\tau_1}(x)}(x)) 
     - f^{k_{\tau_1(x)}}(\tau_1(x))  \\
   &+ f^{k_{\tau_1}(x)}(\tau_1(x)) 
     -\{  f^{k_{\tau_2}(\tau_1(x))}({\tau_2 \circ\tau_1}(x))
        + f^{k_{\tau_1}(x)} (\sigma_A^{k_{\tau_2}(\tau_1(x) ) }({\tau_2 \circ\tau_1}(x)) )\} \\
 =& \rho^f(x,\tau_1)  + f^{l_{\tau_2}(\tau_1(x))} (\sigma_A^{k_{\tau_1}(x)}(\tau_1(x)))  \\
   &+ f^{k_{\tau_1(x)}}(\tau_1(x)) 
     -\{  f^{k_{\tau_2}(\tau_1(x))}({\tau_2 \circ\tau_1}(x))
        + f^{k_{\tau_1}(x)} (\sigma_A^{l_{\tau_2}(\tau_1(x) )}(\tau_1(x)) )\} \\
 =& \rho^f(x,\tau_1)  + f^{l_{\tau_2}(\tau_1(x)) + k_{\tau_1}(x)}(\tau_1(x)) \\ 
   &  -\{  f^{k_{\tau_2}(\tau_1(x))}({\tau_2 \circ\tau_1}(x))
        + f^{k_{\tau_1}(x)} (\sigma_A^{l_{\tau_2}(\tau_1(x) )}(\tau_1(x) ) ) \} \\
=& \rho^f(x,\tau_1)  + f^{l_{\tau_2}(\tau_1(x))}(\tau_1(x))    
  + f^{k_{\tau_1}(x)} (\sigma_A^{l_{\tau_2}(\tau_1(x))}(\tau_1(x)) \\
   &  -\{  f^{k_{\tau_2}(\tau_1(x))}({\tau_2 \circ\tau_1}(x))
        + f^{k_{\tau_1}(x)} (\sigma_A^{l_{\tau_2}(\tau_1(x) )}(\tau_1(x) ) ) \} \\
=& \rho^f(x,\tau_1)  + \rho^f(\tau_1(x), \tau_2).
\hspace{6cm}\qed
\end{align*}
\renewcommand{\qed}{}
\end{proof}
\begin{lemma}\label{eq:rhofzero}
For  $f \in C(X_A,\Z)$, 
we have
$\rho^f(x,\tau) =0$ for all $(x, \tau) \in X_A\times \Gamma_A$
if and only if $f\equiv 0$.
\end{lemma}
\begin{proof}
Suppose that 
$\rho^f(x,\tau) =0$ for all $(x, \tau) \in X_A\times \Gamma_A$.
For any $x = (x_n)_{n\in \N} \in X_A$, 
let $\mu = (x_1,x_2) \in B_2(X_A)$.
By \cite[Lemma 3.2]{MaPacific},
one may find $\tau_\mu \in \Gamma_A$ such that 
\begin{equation*}
\tau_\mu(x) = \sigma_A(x) \quad
\text{ and }
\quad
 k_{\tau_\mu}(x)=0,\qquad l_{\tau_\mu}(x)=1.
\end{equation*}
Hence we have
$
0= \rho^f(x,\tau_\mu) =f(x) + f(\sigma_A(x)) - f(\tau_\mu(x)) = f(x). 
$
\end{proof}
For $\tau \in \Gamma_A$, define
\begin{equation*}
l_{\tau,1}(x)  = l_\tau(\sigma_A(x)) + k_\tau(x) +1, \qquad
k_{\tau,1}(x)  = k_\tau(\sigma_A(x)) + l_\tau(x).
\end{equation*}
\begin{lemma}
$
\sigma_A^{k_{\tau,1}(x)}(\tau(\sigma_A(x))) = \sigma_A^{l_{\tau,1}(x)}(\tau(x)).
$
\end{lemma}
\begin{proof}
We have
\begin{align*}
\sigma_A^{k_{\tau,1}(x)}(\tau(\sigma_A(x)))
 =& \sigma_A^{k_\tau(\sigma_A(x)) + l_\tau(x)}(\tau(\sigma_A(x))) 
 = \sigma_A^{l_\tau(x)}(\sigma_A^{l_\tau(\sigma_A(x)) }(\sigma_A(x)) \\
= & \sigma_A^{l_\tau(\sigma_A(x))+1}(\sigma_A^{k_\tau(x) }(\tau(x))) 
=  \sigma_A^{l_{\tau,1}}(\tau(x)). \hspace{2cm}\qed
\end{align*}
\renewcommand{\qed}{}
\end{proof}
For  $\tau \in \Gamma_A$ and $f \in C(X_A,\Z)$,
we define 
$\varPsi_\tau(f) \in C(X_A,\Z)$ by setting
\begin{align*}
\varPsi_\tau(f)(x) 
= & \sum_{i=0}^{l_{\tau,1}(x)}f(\sigma_A^i(\tau(x)))
 -\sum_{j=0}^{k_{\tau,1}(x)}f(\sigma_A^j(\tau(\sigma_A(x)))) \\
(\, = &f^{l_{\tau,1}(x)}(\tau(x)) - f^{k_{\tau,1}(x)}(\tau(\sigma_A(x))) \, ).
\end{align*}

\begin{lemma}\label{lem:2.8}
For  $f \in C(X_A,\Z)$ and $(x,\tau) \in X_A\times \Gamma_A$,
we have
\begin{enumerate}
\renewcommand{\theenumi}{\roman{enumi}}
\renewcommand{\labelenumi}{\textup{(\theenumi)}}
\item
$\rho^f(x,\tau) - \rho^f(\sigma_A(x),\tau) = f(x) - \varPsi_\tau(f)(x)$. 
\item 
$\rho^f(x,\tau) - \rho^{f\circ\sigma_A}(x,\tau) = f(x) - f(\tau(x))$,
and hence $\rho^{f - f\circ\sigma_A} = \delta_f$. 
\end{enumerate}
\end{lemma}
\begin{proof}
(i) By \eqref{eq:fmn}, we have the following equalities:
\begin{align*}
   & \varPsi_\tau(f)(x) + \rho^f(x,\tau) \\ 
= &f^{l_\tau(\sigma_A(x)) + k_\tau(x) +1}(\tau(x)) - f^{k_\tau(\sigma_A(x)) + l_\tau(x)}(\tau(\sigma_A(x)))
  + f^{l_{\tau}(x)}(x) - f^{k_{\tau}(x)}(\tau(x))\\
= &f^{l_\tau(\sigma_A(x)) + 1}(\sigma_A^{k_\tau(x)}(\tau(x)) )
  - f^{k_\tau(\sigma_A(x))}(\tau(\sigma_A(x)))
  - f^{l_\tau(x)}(\sigma_A^{k_\tau(\sigma_A(x))}(\tau(\sigma_A(x))))
  + f^{l_{\tau}(x)}(x) \\
= &f^{l_\tau(\sigma_A(x)) + 1}(\sigma_A^{l_\tau(x)}(x) )
  - f^{k_\tau(\sigma_A(x))}(\tau(\sigma_A(x)))
  - f^{l_\tau(x)}(\sigma_A^{l_\tau(\sigma_A(x))}(\sigma_A(x)))
  + f^{l_{\tau}(x)}(x) \\
= &f^{l_{\tau}(x)+l_\tau(\sigma_A(x)) + 1}(x) 
  - f^{k_\tau(\sigma_A(x))}(\tau(\sigma_A(x)))
  - f^{l_\tau(x)}(\sigma_A^{l_\tau(\sigma_A(x)) +1}(x)) \\
= &f^{ l_\tau(\sigma_A(x)) }(\sigma_A(x)) + f(x) 
  - f^{k_\tau(\sigma_A(x))}(\tau(\sigma_A(x)))  
=   \rho^f(\sigma_A(x),\tau) + f(x).
\end{align*}
(ii) We also have the following equalities:
\begin{align*}
   &\rho^{f\circ\sigma_A}(x,\tau) + f(x) \\
= &f^{l_\tau(x)}(\sigma_A(x) )
  - f^{k_\tau(x)}(\sigma_A(\tau(x))) + f(x) \\
= &\{ f^{l_{\tau}(x)}(x) - f(x) + f(\sigma_A^{l_\tau(x)}(x))\}
  - \{ f^{k_{\tau}(x)}(\tau(x)) - f(\tau(x)) + f(\sigma_A^{k_\tau(x)}(\tau(x))\} + f(x) \\
= & f^{l_{\tau}(x)}(x) - f^{k_{\tau}(x)}(\tau(x)) + f(\tau(x)) 
=  \rho^f(x,\tau) + f(\tau(x)).\hspace{5cm}
\qed
\end{align*}
\renewcommand{\qed}{}
\end{proof}
Following \cite{MMKyoto} (cf.  \cite{BH}, \cite{Po}),
let us define the  ordered group $(H^A, H^A_+)$
for the one-sided topological Markov shift $(X_A,\sigma_A)$ by
\begin{align*}
H^A  & = C(X_A, \Z) /\{ f - f\circ\sigma_A\mid f \in C(X_A, \Z)\}, \\ 
H^A_{+} & = \{ [f ]  \in H^A \mid f (x) \ge 0 \text{ for all } x \in X_A \}. 
\end{align*}
The two-sided ordered cohomology group $(\bar{H}^A,\bar{H}^A_+)$
have been defined in a similar way in \cite{Po} (cf. \cite{BH}).
In \cite[Lemma 3.1]{MMKyoto},
it was proved that the two ordered cohomology groups
 $(H^A, H^A_+)$ and $(\bar{H}^A,\bar{H}^A_+)$ are actually isomorphic
 as ordered groups.
We have to mention that Boyle-Handelman \cite{BH} proved that the latter 
ordered group is a complete invariant of flow equivalence of two-sided topological Markov shifts.
These two ordered groups together with the Boyle--Handelman result
have played crucial r\^oles in discussions given in \cite{MMKyoto} and \cite{MMETDS}.
\begin{corollary}
For  $f \in C(X_A,\Z)$ and $\tau \in \Gamma_A$,
we have
\begin{enumerate}
\renewcommand{\theenumi}{\roman{enumi}}
\renewcommand{\labelenumi}{\textup{(\theenumi)}}
\item
$[\varPsi_\tau(f)] = [f]$ in $H^A$ for all 
$f \in C(X_A,\Z)$ and $\tau \in \Gamma_A$. 
\item 
$[\rho^f] =[\rho^{f\circ\sigma_A}]$ in $H^1(\Gamma_A,X_A; \Z)$
for all 
$f \in C(X_A,\Z)$. 
\end{enumerate}
\end{corollary}
\begin{proof}
(i)
Put
$\varphi_\tau(f)(x) = \rho^f(x,\tau)$ for $(x, \tau) \in X_A\times \Gamma_A$.
The above lemma says that 
\begin{equation*}
f - \varPsi_\tau(f) = \varphi_\tau(f) - \varphi_\tau(f) \circ\sigma_A
\end{equation*}
so that we have
$[\varPsi_\tau(f)] = [f]$ in $H^A$.

(ii) The assertion follows directly from Lemma \ref{lem:2.8} (ii).
\end{proof}
\begin{proposition}
The homomorphism
\begin{equation}
\rho: \quad f \in C(X_A,\Z) \longrightarrow 
\rho^f \in Z^1(\Gamma_A,X_A;\Z) \label{eq:rhoh}
\end{equation}
extends to an injective  homomorphism
\begin{equation*}
\rho_H: \quad [f] \in H^A
\longrightarrow [\rho^f] \in H^1(\Gamma_A,X_A;\Z)
\end{equation*}
of groups such that 
$\rho_H([1_{X_A}]) = [d^A]$.
\end{proposition}
\begin{proof}
For $f \in C(X_A,\Z)$ and $(x,\tau) \in X_A\times\Gamma_A$, 
Lemma \ref{lem:2.8} (ii) tells us that 
$$
\rho^{f - f \circ \sigma_A}(x,\tau) 
= \rho^f(x,\tau) - \rho^{f\circ\sigma_A}(x,\tau) =\delta_f(x,\tau).
$$
As $\delta_f $ is a coboundary, the homomorphism \eqref{eq:rhoh}
is well-defined.

We will next show that 
$\rho_H: \, [f] \in H^A
\longrightarrow [\rho^f] \in H^1(\Gamma_A,X_A;\Z)
$
is injective.
Suppose that 
$[\rho^f] =0$ in $ H^1(\Gamma_A,X_A;\Z)$.
Take $g \in C(X_A, \Z)$ such that 
$\rho^f(x,\tau) = g(x) - g(\tau(x))$ 
for all $(x,\tau) \in X_A\times\Gamma_A$.
By Lemma \ref{lem:2.8} (ii), we have
$$
\rho^g(x,\tau) - \rho^{g\circ\sigma_A}(x,\tau) = g(x) - g(\tau(x)) =\rho^f(x,\tau)
$$
so that
$\rho^{f - (g - g\circ\sigma_A)}(x,\tau)=0$.
By Lemma \ref{eq:rhofzero},
we have $f \equiv g- g\circ\sigma_A$.
\end{proof}


\section{Continuous orbit equivalence and cohomology groups}
Throughout this section,
we assume that
$(X_A, \sigma_A)$ and $(X_B,\sigma_B)$ 
are continuously orbit equivalent
via a homeomorphism $h:X_A\longrightarrow X_B$
with cocycle functions 
$c_1 = l_1 - k_1$ 
and
$c_2= l_2 - k_2$. 
Define $\xi_h: \Homeo(X_A) \longrightarrow \Homeo(X_B)$
by setting
$\xi_h(\tau) = h \circ\tau\circ h^{-1}, \, \tau\in \Homeo(X_A)$.  
It has been proved that 
$ \xi_h(\Gamma_A) = \Gamma_B$ (\cite[Proposition 5.4]{MaPacific}).
In the proof of \cite[Proposition 5.4]{MaPacific}, 
we actually showed the following lemma.
\begin{lemma}\label{lem:htauh}
For $\tau \in \Gamma_A$, put
$n = l_\tau(x), m = k_\tau(x)$. We then have
$$
l_{\xi_h(\tau)} (h(x)) = k_1^m(\tau(x)) + l_1^n(x),
\qquad
k_{\xi_h(\tau)} (h(x)) = l_1^m(\tau(x)) + k_1^n(x),
\qquad 
x \in X_A.
$$
Hence we have
\begin{equation}\label{eq:3.1}
d_{\xi_h(\tau)} (h(x)) = c_1^n(x) - c_1^m(\tau(x)), \qquad x \in X_A.
\end{equation}
\end{lemma}

For a one-cocycle
$\rho: X_A \times \Gamma_A \longrightarrow \Z$,
define 
$\varPhi_h(\rho): X_B \times \Gamma_B \longrightarrow \Z$
by setting
\begin{equation*}
\varPhi_h(\rho)(y,\varphi) = 
\rho(h^{-1}(y), \xi_{h^{-1}}(\varphi) )\in \Z,
\qquad (y,\varphi) \in X_B\times \Gamma_B.
\end{equation*}
Recall that 
 a homomorphism
$\varPsi_{h^{-1}}: C(X_A,\Z) \rightarrow C(X_B,\Z)$
is defined in \eqref{eq:Psihf}
for $h^{-1}: X_B \rightarrow X_A$.
We provide a lemma.
\begin{lemma}\label{lem:coecoh}
Keep the above notation. 
We have
\begin{equation}
\varPhi_h(\rho^f)= \rho^{\varPsi_{h^{-1}}(f)} \qquad \text{ for } f \in C(X_A,\Z).
\end{equation}
\end{lemma}
\begin{proof}
For $(y,\varphi) \in X_B\times\Gamma_B$,
put $n =k_\varphi(y), m= l_\varphi(y)$.
We have by using Lemma \ref{lem:htauh} 
\begin{align*}
\varPhi_h(\rho^f)(y, \varphi)
= & f^{l_{\xi_{h^{-1}}(\varphi)}(h^{-1}(y))}(h^{-1}(y))
   -f^{k_{\xi_{h^{-1}}(\varphi)}(h^{-1}(y))}(\xi_{h^{-1}}(\varphi)(y)) \\
= & f^{k_2^n(\varphi(y)) + l_2^m(y)}(h^{-1}(y))
   -f^{k_2^m(y) + l_2^n(\varphi(y))}(h^{-1}(\varphi(y))) \\
= & f^{l_2^m(y)}(h^{-1}(y)) + f^{k_2^n(\varphi(y))}(\sigma_A^{l_2^m(y)}(h^{-1}(y))) \\
  &- \{ f^{ l_2^n(\varphi(y))}(h^{-1}(\varphi(y))) 
       + f^{k_2^m(y)}(\sigma_A^{ l_2^n(\varphi(y))}(h^{-1}(\varphi(y)))) \} \\
= & f^{l_2^m(y)}(h^{-1}(y))+f^{k_2^n(\varphi(y))}(\sigma_A^{k_2^m(y)}(h^{-1}(\sigma_B^m(y)))) \\
   & - \{ f^{ l_2^n(\varphi(y))}(h^{-1}(\varphi(y))) 
       + f^{k_2^m(y)}(\sigma_A^{ k_2^n(\varphi(y))}(h^{-1}(\sigma_B^n(\varphi(y))))) \}.
 \end{align*}
On the other hand, by using \cite[Lemma 4.5]{MaPAMS2016}, we have
\begin{align*}
\rho^{\varPsi_{h^{-1}}(f)}(y,\varphi)
=& \varPsi_{h^{-1}}(f)^{m}(y) - \varPsi_{h^{-1}}(f)^{n}(\varphi(y)) \\ 
=& f^{l_2^m(y)}(h^{-1}(y)) - f^{k_2^m(y)}(h^{-1}(\sigma_B^m(y))) \\
 &-\{ f^{l_2^n(\varphi(y))}(h^{-1}(\varphi(y))) 
- f^{k_2^n(\varphi(y))}(h^{-1}(\sigma_B^n(\varphi(y))))\}.  
 \end{align*}
As
$\sigma_B^m(y) = \sigma_B^n(\varphi(y))$,
we see that
\begin{align*}
 & \{ f^{k_2^n(\varphi(y))}(\sigma_A^{k_2^m(y)}(h^{-1}(\sigma_B^m(y)))) 
      + f^{k_2^m(y)}(h^{-1}(\sigma_B^m(y))) \} \\
  &- \{ f^{k_2^m(y)}(\sigma_A^{ k_2^n(\varphi(y))}(h^{-1}(\sigma_B^n(\varphi(y))))) 
   + f^{k_2^n(\varphi(y))}(h^{-1}(\sigma_B^n(\varphi(y)))) \}   \\
= &  f^{k_2^m(y)+ k_2^n(\varphi(y))}(h^{-1}(\sigma_B^m(y)))  
- f^{k_2^m(y) + k_2^n(\varphi(y))}(h^{-1}(\sigma_B^n(\varphi(y))))
=0 
\end{align*}
so that we conclude that
$
\varPhi_h(\rho^f)(y, \varphi)
= \rho^{\varPsi_{h^{-1}}(f)}(y,\varphi). \hspace{4cm} \qed
$ 
\renewcommand{\qed}{}
\end{proof}
We will show the following proposition.
\begin{proposition}\label{prop:COECOHOM}
\hspace{6cm}
\begin{enumerate}
\renewcommand{\theenumi}{\roman{enumi}}
\renewcommand{\labelenumi}{\textup{(\theenumi)}}
\item
If $(X_A,\sigma_A)$ and 
$(X_B,\sigma_B)$
are continuously orbit equivalent,
then 
there exists an isomorphism
$
\varPhi_h:H^1(\Gamma_A,X_A; \Z)
\longrightarrow
H^1(\Gamma_B,X_B; \Z)$
of ordered groups
such that 
$
\varPhi_h(\rho^f) = \rho^{\varPsi_{h^{-1}}(f)}, f \in C(X_A,\Z).
$
\item
If in particular 
$(X_A,\sigma_A)$ and 
$(X_B,\sigma_B)$
are strongly continuous orbit equivalent,
the isomorphism 
$\varPhi_h: H^1(\Gamma_A,X_A;{\Bbb Z}) \longrightarrow H^1(\Gamma_B,X_B;{\Bbb Z})$
satisfies 
$\varPhi_h([d^A]) = [d^B].$
\end{enumerate}
\end{proposition}
\begin{proof}
(i)
For $\varphi_1,\varphi_2 \in \Gamma_B$,
it follows that
\begin{align*}
\varPhi_h(\rho)(y,\varphi_2\circ \varphi_1)
& = \rho(h^{-1}(y), 
h^{-1}\circ \varphi_2\circ h\circ h^{-1}\varphi_1 \circ h ) \\
& = \rho(h^{-1}(y), h^{-1}\circ \varphi_1\circ h) 
+   \rho(h^{-1}\circ \varphi_1\circ h (h^{-1}(y)),h^{-1}\circ \varphi_2\circ h) \\
& =\Phi_h(\rho)(y,\varphi_1) + \Phi_h(\rho)(\varphi_1(y),\varphi_2)
\end{align*}
so that
$\varPhi_h(\rho)\in Z^1(\Gamma_B,X_B;\Z)$.
We also have for 
$g \in C(X_A, \Z)$,
\begin{equation*}
\varPhi_h(\delta_g)(y,\varphi)
 = \delta_g(h^{-1}(y), \xi_{h^{-1}}(\varphi) ) 
 = (g \circ h^{-1})(y) - (g \circ  h^{-1})( \varphi(y))
\end{equation*}
so that 
$\varPhi_h(\delta_g) = \delta_{g\circ h^{-1}}\in B^1(\Gamma_B,X_B;\Z)$.
Therefore we have a homomorphism
$\varPhi_h: Z^1(\Gamma_A,X_A;\Z) \longrightarrow Z^1(\Gamma_B,X_B;\Z)$
such that
$\varPhi_h(B^1(\Gamma_A,X_A; \Z) ) \subset B^1(\Gamma_B,X_B; \Z)$
and
$\varPhi_h(Z^1_+(\Gamma_A,X_A; \Z) ) \subset Z^1_+(\Gamma_B,X_B; \Z)$.
It induces a homomorphism
$ H^1(\Gamma_A,X_A;\Z) \longrightarrow H^1(\Gamma_B,X_B;\Z)$. 
It is also written  $\varPhi_h$
 and satisfies 
$\varPhi_h( H^1_+(\Gamma_A,X_A;\Z)) \subset H^1_+(\Gamma_B,X_B;\Z)$.
Similarly we 
$\varPhi_{h^{-1}}: H^1(\Gamma_B,X_B;\Z) \longrightarrow H^1(\Gamma_A,X_A;\Z)$
for $h^{-1}:X_B\longrightarrow X_A$.
It is direct to see that 
$\varPhi_{h^{-1}} = \varPhi_h^{-1}$.
Hence we have
 an order preserving isomorphism from
$H^1(\Gamma_A,X_A; \Z)$
to
$H^1(\Gamma_B,X_B; \Z)$.

(ii)
We further assume that a homeomorphism
$h:X_A \longrightarrow X_B$ 
 gives rise to a strongly continuous orbit equivalence between
$(X_A,\sigma_A)$ and 
$(X_B,\sigma_B)$.
We will prove $\varPhi_h([d^A]) = [d^B]$.
For $(y,\varphi) \in X_B\times\Gamma_B$, by putting
$n = k_\varphi(y), m= l_\varphi(y)$, we have the equalities
\begin{equation*}
l_{\xi_{h^{-1}}(\varphi)}(h^{-1}(y))  = k_2^n(\varphi(y)) + l_2^m(y), \qquad
k_{\xi_{h^{-1}}(\varphi)}(h^{-1}(y))  = k_2^m(y) + l_2^n(\varphi(y)) 
\end{equation*}
by Lemma \ref{lem:htauh}.
It then follows that
\begin{align*}
\varPhi_h(d^A)(y,\varphi) 
=& d^A(h^{-1}(y), \xi_{h^{-1}}(\varphi) ) 
= d_{ \xi_{h^{-1}}(\varphi)}(h^{-1}(y)) \\
=& l_{\xi_{h^{-1}}(\varphi)}(h^{-1}(y)) -
   k_{\xi_{h^{-1}}(\varphi)}(h^{-1}(y))  
= c_2^m(y) - c_2^n(\varphi(y)). 
\end{align*}
Now 
$h:X_A \longrightarrow X_B$  gives rise to a strongly continuous orbit equivalence,
so that $c_2 = 1 + b_2 - b_2\circ \sigma_B$
for some $b_2 \in C(X_B,\Z)$.
Hence by \cite[Lemma 5.2]{MaJOT},
we have
\begin{align*}
c_2^m(y) - c_2^n(\varphi(y))
= & \{m + b_2(y) - b_2(\sigma_B^m(y)) \} -
\{n + b_2(\varphi(y)) - b_2(\sigma_B^n(\varphi(y))) \} \\
= & m -n  + b_2(y) -  b_2(\varphi(y)) +\{
b_2(\sigma_B^m(y))  - b_2(\sigma_B^n(\varphi(y))) \} \\
=& d^B(y,\varphi) + \delta_{b_2}(y,\varphi).  
\end{align*}
Consequently we see that
$
\varPhi_h(d^A)(y,\varphi) = d^B(y,\varphi) + \delta_{b_2}(y,\varphi)  
$
so that
$\varPhi_h([d^A]) = [d^B]$.
\end{proof}
Therefore we have
\begin{proposition} Keep the above notation.
\begin{enumerate}
\renewcommand{\theenumi}{\roman{enumi}}
\renewcommand{\labelenumi}{\textup{(\theenumi)}}
\item
If one-sided topological Markov shifts
$(X_A,\sigma_A)$ and 
$(X_B,\sigma_B)$
are continuously orbit equivalent,
then there exist order preserving isomorphisms
$\varPsi_{h^{-1}}: H^A \rightarrow H^B$ and 
$\varPhi_h: H^1(\Gamma_A,X_A; \Z) \rightarrow H^1(\Gamma_B,X_B;\Z)$
such that the diagram
$$
\begin{CD}
H^A @>\rho_H >> H^1(\Gamma_A,X_A;\Z)  \\
@V{\varPsi_{h^{-1}}}VV  @VV{\varPhi_{h}}V \\
H^B @> \rho_H >> H^1(\Gamma_B,X_B;\Z) 
\end{CD}
$$
commutes.
\item
If in particular 
$(X_A,\sigma_A)$ and 
$(X_B,\sigma_B)$
are strongly continuous orbit equivalent,
in the above diagram, we have
$\varPsi_{h^{-1}}([1_{X_A}]) = [1_{X_B}]$ and $\varPhi_h([d^A]) = [d^B].$
\end{enumerate}
\end{proposition}

We next prove the converse of Proposition \ref{prop:COECOHOM} (ii)
in the followng proposition.
\begin{proposition}\label{prop:scoecoho}
One-sided topological Markov shifts
$(X_A,\sigma_A)$ and 
$(X_B,\sigma_B)$
are strongly continuous orbit equivalent
if and only if
there exists a homeomorphism
$h:X_A\longrightarrow X_B$ such that 
$h\circ \Gamma_A\circ h^{-1} = \Gamma_B$ 
and the induced isomorphism 
$\varPhi_h: H^1(\Gamma_A,X_A;{\Bbb Z}) \longrightarrow H^1(\Gamma_B,X_B;{\Bbb Z})$
of the cohomology groups satisfies
$\varPhi_h([d^A]) = [d^B].$
\end{proposition}
\begin{proof}
It suffices to show the if part.
Assume that there exists a homeomorphism
$h:X_A\longrightarrow X_B$ such that 
$\xi_h(\Gamma_A) = \Gamma_B$ 
and the induced isomorphism 
$\varPhi_h: H^1(\Gamma_A,X_A;\Z) \longrightarrow H^1(\Gamma_B,X_B;\Z)$
of the cohomology groups satisfies
$\varPhi_h([d^A]) = [d^B].$
Hence we have
$\varPhi_h(d^A) - d^B \in B^1(\Gamma_B, X_B;\Z).$
Take
$b_2 \in C(X_B,\Z)$ such that 
\begin{equation}
\varPhi_h(d^A)(y,\varphi) - d^B(y,\varphi) = b_2(y) - b_2(\varphi(y)), \qquad
(y, \varphi) \in X_B\times \Gamma_B.
\label{eq:scoecoho1}
\end{equation}
By putting $ x = h^{-1}(y), \tau =\xi_{h^{-1}}(\varphi)$ and $ b_1(x) = b_2(h(x))$,
we have 
\begin{gather*}
\varPhi_h(d^A)(y,\varphi) = d^A(h^{-1}(y), h^{-1}\circ \varphi\circ h) = d^A(x,\tau) = d_\tau(x), 
\\
d^B(y,\varphi) = d^B(h(x),h\circ\tau\circ h^{-1}) =d_{\xi_h(\tau)}(h(x)),
\qquad 
b_1(\tau(x)) = b_2(\varphi(y)).
\end{gather*}
Hence the equality \eqref{eq:scoecoho1} implies
\begin{equation*}
d_\tau(x) - d_{\xi_h(\tau)}(h(x)) = b_1(x) - b_1(\tau(x)), \qquad x \in X_A, \tau\in \Gamma_A.
\end{equation*}
For $\mu \in B_2(X_A)$, 
take $\tau_\mu \in \Gamma_A$ as in \cite[Lemma 3.2]{MaPacific}.
Let $\{ \mu^{(1)}, \dots, \mu^{(M)} \}= B_2(X_A)$.
Put
$$
\tau_{(i)} = h\circ \tau_{\mu^{(i)}}\circ h^{-1} \in \Gamma_B,
\qquad i=1,\dots,M.
$$
Following the proof of \cite[Lemma 3.2]{MaPacific},
we have
$$
h(\sigma_A(x)) = h(\tau_{\mu^{(i)}}(x)) = \tau_{(i)}(h(x)), \qquad x \in U_{\mu^{(i)}}.
$$
Since
$
\tau_{(i)} \in \Gamma_B$,
one may find
$l_{\tau_{(i)}}, k_{\tau_{(i)}} \in C(X_B, \Zp)
$ such that 
$$
\sigma_B^{k_{\tau_{(i)}}(y)}(\tau_{(i)}(y)) = \sigma_B^{l_{\tau_{(i)}}(y)}(y),
\qquad y \in X_B.
$$
Define
$
k_1^h(x) = k_{\tau_{(i)}}(h(x)), \,
l_1^h(x) = l_{\tau_{(i)}}(h(x))
$
for
$
 x \in  U_{\mu^{(i)}}. 
$
For $y = h(x) \in h(U_{\mu^{(i)}})$, we then have
$$
\sigma_B^{k_1^h(x)}(h(\sigma_A(x))) 
= \sigma_B^{l_1^h(x)}(h(x)), \qquad x \in U_{\mu^{(i)}}.
$$
It then follows that for  
$ x \in  U_{\mu^{(i)}}$
\begin{align*}
l_1^h(x) - k_1^h(x)
=& d_{h\circ \tau_{\mu^{(i)}}\circ h^{-1}}(h(x))    
= d_{\tau_{\mu^{(i)}}}(x)  - b_1(x) + b_1(\tau_{\mu^{(i)}}(x)) \\
=& l_{\tau_{\mu^{(i)}}}(x) - k_{\tau_{\mu^{(i)}}}(h(x))  - b_1(x) + 
b_1(\tau_{\mu^{(i)}}(x)).  
\end{align*}
By the construction of $\tau_{\mu^{(i)}}$ as in \cite[Lemma 3.2]{MaPacific},
we know that $\tau_{\mu^{(i)}}(x) = \sigma_A(x)$
and 
$l_{\tau_{\mu^{(i)}}}(x)=1,
 k_{\tau_{\mu^{(i)}}}(x) =0$.
Therefore  we have
\begin{equation}
l_1^h(x) - k_1^h(x) = 1 - b_1(x) + b_1(\sigma_A(x)), \qquad x \in  U_{\mu^{(i)}}.
\label{eq:scoecoho2}
\end{equation} 
Hence the equality \eqref{eq:scoecoho2} holds for all $x \in X_A$,
showing that 
$(X_A,\sigma_A)$ and 
$(X_B,\sigma_B)$
are strongly continuous orbit equivalent.
\end{proof}

\section{Contnuous orbit equivalence and cocycle subgroup $\Gamma_{A,f}$}
In this section, we will give the proof of  Theorem \ref{thm:main1}.
Before completeing  a proof of Theorem \ref{thm:main1},
we provide a lemma and a proposition in the following way.
\begin{lemma} \label{lem:taucommute}
If a homeomorphism $h$ on $X_A$ commetes with all elements of $\Gamma_A$,
then $h=\id$.
\end{lemma}
\begin{proof}
Suppose that $h \ne \id.$
Since $h:X_A\longrightarrow X_A$ is a nontrivial homeomorphism,
there exist cylinder sets $U_\mu, U_\nu \subset X_A$
such that 
\begin{equation*}
U_\mu \cap U_\nu = \emptyset, \qquad
h(U_\mu) \subset U_\nu, \qquad
U_\nu \backslash h(U_\mu) \ne \emptyset.
\end{equation*}
We define open sets $U = h(U_\mu), Y = U_\nu \backslash h(U_\mu)$
in $X_A$
and take $x \in U$.
By \cite[Lemma 2.1]{MaIsrael2015}, 
there exist a clopen set $V$ of $X_A$ and an element 
$\tau \in \Gamma_A$
such that 
\begin{equation*}
x \in V \subset U, \qquad 
\tau(V) \subset Y, \qquad 
\tau|_{(V\cup\tau(V))^c} =\id.
\end{equation*}
Now we have
$h^{-1}(x) \in U_\mu$ and
$V \cup \tau(V) \subset U_\nu$.
Since 
$U_\nu \cap U_\mu = \emptyset$,
we have
$
h^{-1}(x) \in (V\cup\tau(V))^c
$
so that $\tau(h^{-1}(x)) = h^{-1}(x)$, 
and hence  
\begin{equation*}
(h\circ \tau)(h^{-1}(x))= x \in V \subset h(U_\mu).
\end{equation*}
On the other hand,
$
(\tau\circ h)(h^{-1}(x))= \tau(x) \in \tau(V) \subset U_\nu\backslash h(U_\mu).
$
We thus conclude that 
\begin{equation*}
(h\circ \tau)(h^{-1}(x))\ne (\tau\circ h)(h^{-1}(x))
\end{equation*}
and hence 
$h\circ \tau \ne \tau\circ h$, a contradiction.
\end{proof}

\begin{proposition}[{\cite[Theorem 7.2]{MaIsrael2015}, cf. \cite{MatuiCrelle}}]
\label{prop:implement}
Suppose that there exists an isomorphism
$\xi: \Gamma_A\longrightarrow \Gamma_B$ 
of groups. Then there exists a unique homeomorphism
$h:X_A\longrightarrow X_B$ such that
\begin{equation}
\xi(\tau) = h\circ\tau\circ h^{-1}, \qquad \tau \in \Gamma_A. \label{eq:xihtau}
\end{equation}
\end{proposition}
\begin{proof}
By \cite[Theorem 7.2]{MaIsrael2015}, we know that 
there exists a homeomorphism satisfying
\eqref{eq:xihtau}.
It suffices to show its uniqueness of $h$.
Suppose that there exist homeomorphisms
$h_i : X_A\longrightarrow X_B, i=1,2$ satisfying \eqref{eq:xihtau}.
Since the homeomorphism $h_2^{-1}\circ h_1$ on $X_A$ commutes with all elements of 
$\Gamma_A$,
we have $h_2 = h_1$ by Lemma \ref{lem:taucommute}.
\end{proof}
By the above proposition and the proof of \cite[Proposition 5.3]{MaPacific},
 an isomorphism
$\xi: \Gamma_A\longrightarrow \Gamma_B$
yields a unique homeomorphism $h: X_A\longrightarrow X_B$
giving rise to a continuous orbit equivalence between 
$(X_A, \sigma_A)$ and $(X_B,\sigma_B)$.
The homeomorphism $h: X_A\longrightarrow X_B$
defines a homomorphism $\varPsi_h: C(X_B,\Z) \longrightarrow C(X_A,\Z)$ by
the formula \eqref{eq:Psihf}.

Recall that for a function $f \in C(X_A, \Z)$ 
the cocycle subgroup $\Gamma_{A,f}$ is defined by \eqref{eq:cocycle}.
The following proposition shows that 
$\Gamma_{A,f}$ is actually a group.
\begin{proposition}\label{prop:cocyclesubgroup}
$\Gamma_{A,f} $ is a subgroup of $\Gamma_A$.
\end{proposition}
\begin{proof}
For $\tau_1,\tau_2 \in \Gamma_A$, the formula in Lemma \ref{lem:rhoftau21}
tells us that 
$\tau_1, \tau_2 \in \Gamma_{A,f}$ implies
$\tau_2\circ \tau_1 \in \Gamma_{A,f}$.
By Lemma \ref{lem:rhoftau21} again, 
we have
\begin{equation*}
0= \rho^f(x,\id) =\rho^f(x, \tau^{-1}\circ\tau) 
= \rho^f(x,\tau) + \rho^f(\tau(x), \tau^{-1}), \qquad x \in X_A 
\end{equation*}
so that 
$\rho^f(x,\tau^{-1}) = - \rho^f(\tau^{-1}(x), \tau)$.
Hence we have that  
$\tau\in \Gamma_{A,f}$ implies
$\tau^{-1} \in \Gamma_{A,f}$.
\end{proof}
For $\tau \in \Gamma_A$, we see that 
$\tau \in \Gamma_{A,f}$ if and only if 
$f^{k_{\tau}(x)}(\tau(x)) = f^{l_\tau(x)}(x), x \in X_A.$
By definition, we know that 
$\Gamma _{A, m f} = \Gamma_{A, f}$ for a nonzero integer $m \in \Z$.

\begin{proposition}\label{prop:xihAB}
Keep the above situation.
If there exists an isomorphism
$\xi: \Gamma_A\longrightarrow \Gamma_B$,
then we have
\begin{equation*}
\xi_h(\Gamma_{A, \varPsi_h(f)}) = \Gamma_{B,f}, \qquad f \in C(X_B,\Z).
\end{equation*}
\end{proposition}
\begin{proof}
By Lemma \ref{lem:coecoh},
 we have
\begin{equation*}
\rho^{\varPsi_h(f)}(x,\tau)
=\varPhi_{h^{-1}}(\rho^f)(x,\tau)
=\rho^f(h(x), \xi_h(\tau))
  \qquad \text{ for } f \in C(X_A,G)
\end{equation*}
Hence
$\rho^{\varPsi_h(f)}(x,\tau) =0$ for all $x \in X_A$
if and only if 
$\rho^f(y, \xi_h(\tau)) =0$ for all $y \in X_B.$
This shows that 
$\tau \in \Gamma_{A, \varPsi_h(f)}$
if and only if
$\xi_h(\tau)  \in \Gamma_{B,f}.$
\end{proof}

\medskip

Now we are in position to complete the proof of Theorem \ref{thm:main1}.

\medskip

 {\it Proof of Theorem \ref{thm:main1}:} 
 The assertion (i)   was proved in \cite[Theorem 3.2]{MaMZ}.
The assertion (ii) follows from  Proposition \ref{prop:xihAB}.
The assertion (iii) is proved in Proposition \ref{prop:COECOHOM} (i).

 We therefore completed the proof of Theorem \ref{thm:main1}. \qed

\section{Strongly continuous orbit equivalence and coboundary subgroups $\Gamma_A^b$}
In this section, we restrict our interest to strongly continuous orbit equivalence.
Recall that for a function $b \in C(X_A, \Z)$ 
the coboundary subgroup $\Gamma_A^b$ is defined by \eqref{eq:coboundarysub}.
\begin{lemma}\label{lem:1b}
For $b \in C(X_A,\Z)$, define $1_b \in C(X_A,\Z)$ 
by $1_b = 1 -b + b\circ\sigma_A$.
Then we have
$
\Gamma_A^b =\Gamma_{A,1_b}$
the cocycle subgroup for $1_b$.
\end{lemma}
\begin{proof}
We have the equalities
\begin{equation*}
\rho^{1 -(b - b\circ \sigma_A)}(x,\tau)
=\rho^1(x,\tau) -( \rho^{b}(x,\tau)- \rho^{b\circ\sigma_A}(x,\tau) ).
\end{equation*}
By Lemma \ref{lem:2.8} (ii) and the identity
$\rho^1(x,\tau) = d_\tau(x)$, we have 
$\rho^{1 -(b -b\circ\sigma_A)}(x,\tau)=0$
if and only if
$d_\tau(x) = b(x) - b(\tau(x))$.
Hence we have
$\tau \in
\Gamma_{A,1_b}
$ if and only if
$\tau \in \Gamma_A^b$.
\end{proof}
Since a cocycle subgroup is a subgroup of $\Gamma_A$,
the coboundary subgroup $\Gamma_A^b$ is also a subgroup of $\Gamma_A$.
If we take $b\equiv m$ a constant integer and hence $1_b \equiv 1$, 
we see that 
$\Gamma_{A,1} = \Gamma_A^m= \Gamma_A^{\AF}$:
the AF full group of $(X_A, \sigma_A)$.
By definition, we know that 
$\Gamma _A^{m +b} = \Gamma_A^b$ for a nonzero integer $m \in \Z$.

We will present an example of $\Gamma_A^b$ with nontrivial $b$.
\begin{example}
Let $A$ be an irreducible non permutation matrix with entries in $\{0,1\}$.
Take $\mu = ab \in B_2(X_A)$  with $a\ne b$.
Let us define a homeomorphism
$\tau_\mu $ on $X_A$ by 
\begin{equation*}
\tau_\mu(x) := 
\begin{cases}
\sigma_A(x) & \text{ for } x \in U_\mu, \\
a x & \text{ for } x \in U_b, \\
x & \text{ else},
\end{cases}
\end{equation*}
and continuous functions
\begin{equation*}
k_{\tau_\mu}(x) := 
\begin{cases}
0 & \text{ for } x \in U_\mu, \\
1 & \text{ for } x \in U_b, \\
0 & \text{ else},
\end{cases}
\qquad
l_{\tau_\mu}(x) := 
\begin{cases}
1 & \text{ for } x \in U_\mu, \\
0 & \text{ for } x \in U_b, \\
0 & \text{ else}.
\end{cases}
\end{equation*}
It is straightforward to see that 
$
\sigma_A^{k_{\tau_\mu}(x)}(\tau_\mu(x)) = \sigma_A^{l_{\tau_\mu}(x)}(x), x \in X_A,
$ 
so that $\tau_\mu \in \Gamma_A$.
Put $b_{\mu} = l_{\tau_\mu}$.
As $k_{\tau_\mu} = l_{\tau_\mu}\circ \tau_{\mu}$,  
we know  
$
d_{\tau_\mu} = b_{\mu} - b_{\mu}\circ \tau_\mu,
$
so that 
$\tau_\mu \in \Gamma_A^{b_\mu}.$
 \end{example}

Using Lemma \ref{lem:htauh}, we may show
 the following proposition.
\begin{proposition}\label{prop:3.2}
Suppose that 
$(X_A,\sigma_A) \underset{\scoe}{\sim}(X_B,\sigma_B)$
via a homeomorphism $h:X_A\longrightarrow X_B$
with coboundary maps $b_1 \in C(X_A,\Z)$ and
$b_2 \in C(X_B,\Z)$ such that
\begin{equation*}
c_1 = 1 + b_1 - b_1\circ\sigma_A \qquad 
c_2 = 1 + b_2 - b_2\circ \sigma_B.
\end{equation*} 
Then we have
\begin{enumerate}
\renewcommand{\theenumi}{\roman{enumi}}
\renewcommand{\labelenumi}{\textup{(\theenumi)}}
\item
$\xi_h(\Gamma_A^f) =\Gamma_B^{f\circ h^{-1} + b_1\circ h^{-1}}, \qquad f \in C(X_A,\Z),$
\item
$
\xi_{h^{-1}}(\Gamma_B^g) =\Gamma_A^{g\circ h + b_2\circ h}, \qquad g \in C(X_B,\Z).
$ 
\end{enumerate}  
\end{proposition}
\begin{proof}
(i) 
By \cite[Lemma 5.2]{MaJOT}, we know the equality
\begin{equation}
c_1^n(x) = n + b_1(x) - b_1(\sigma_A^n(x)), \qquad x \in X_A, \, \, n \in \Zp.
\end{equation}
For $\tau \in \Gamma_A^f$, put $n =l_\tau(x), m = k_\tau(x), x \in X_A.$
By Lemma \ref{lem:htauh}, we have
\begin{align*}
d_{\xi_h(\tau)} (h(x)) 
=&  c_1^n(x) - c_1^m(\tau(x)) \\
= & \{n + b_1(x) - b_1(\sigma_A^n(x))\} - \{m + b_1(\tau(x)) - b_1(\sigma_A^m(\tau(x)))\} \\
= & n -m + b_1(x)  - b_1(\tau(x)) \\
= & f(x) - f(\tau(x)) + b_1(x) - b_1(\tau(x)) \\
= & (f + b_1)(x) - (f+ b_1)(\tau(x)), 
\end{align*}
so tha we have 
  $\xi_h(\tau) \in \Gamma_B^{f\circ h^{-1} + b_1\circ h^{-1}}$
  and hence
$\xi_h(\Gamma_A^f) \subset \Gamma_B^{f\circ h^{-1} + b_1\circ h^{-1}}$.
We similarly have
$\xi_{h^{-1}}(\Gamma_B^g) \subset \Gamma_A^{g\circ h + b_2\circ h}$.
We thus have the inclusion relations:
$$
\Gamma_A^f = \xi_{h^{-1}}(\xi_h(\Gamma_A^f) ) 
           \subset \xi_{h^{-1}}( \Gamma_B^{f\circ h^{-1} + b_1\circ h^{-1}})
           \subset \Gamma_A^{f + b_1 + b_2\circ h}
$$
By \cite[Lemma 5.4]{MaJOT}, the function $b_1 + b_2\circ h$
is a constant which we denote by $N_h$.
As 
$\Gamma_A^{f + N_h} = \Gamma_A^{f},
$
we conclude that  
$\xi_h(\Gamma_A^f) =\Gamma_B^{f\circ h^{-1} + b_1\circ h^{-1}}$
and 
$
\xi_{h^{-1}}(\Gamma_B^g) =\Gamma_A^{g\circ h + b_2\circ h}.
$ 
\end{proof}
The following lemma is needed to show 
 the converse of Proposition \ref{prop:3.2}.
A matrix $A=[A(i,j)]_{i,j=1}^N$ with $N\ge 2$
 is said to be primitive if there exists a positive integer $n$ such that 
$A^n(i,j)$ is positive for all $i,j =1,\dots,N$.
A primitive matrix is irreducible and not a permutation matrix.
\begin{lemma}\label{lem:3.5}
Assume that the matrix $A$ is primitive.
Suppose that there exist a homeomorphism $h:X_A\longrightarrow X_B$
that gives rise to
$(X_A,\sigma_A) \underset{\coe}{\sim}(X_B,\sigma_B)$
and a continuous function  $b_1 \in C(X_A,\Z)$ such that 
$
\xi_h (\Gamma_A^{\AF}) \subset \Gamma_B^{b_1\circ h^{-1}}.
$
Then the function $g_1(x) := c_1(x) - (b_1(x) - b_1(\sigma_A(x)) ), x \in X_A$
is a constant.
\end{lemma}
\begin{proof}
Suppose that $g_1$ is not constant.
A point $x \in X_A$ is said to be eventually periodic
if there exist
$r, s \in \N$ with $r>s $ satisfying 
 $\sigma_A^r(x) = \sigma_A^s(x)$.
As the matrix $A$ is primitive,  
the set $X_A^{\nep}$ of non eventually periodic points of $X_A$ 
is dense in $X_A$, 
and also we may find $z \in X_A^{\nep}$ and $\tau\in \Gamma_A^{\AF}$ such that 
$g_1(z) \ne g_1(\tau(z))$.
As we may choose  $k \in \N$ such that 
$\sigma_A^k(z) = \sigma_A^k(\tau(z)),$
the set 
\begin{equation*}
\mathbb{K} =
\{ k \in \N \mid \exists x \in X_A^{\nep}, \exists \tau \in \Gamma_A^{\AF} ;
g_1(x) \ne g_1(\tau(x)), \sigma_A^k(x) = \sigma_A^k(\tau(x))
\}
\end{equation*} 
is not empty.
We put
$K_0 = \min{\mathbb{K}}.$ 
Take $x_0 \in  X_A^{\nep}$ and
$\tau _0 \in \Gamma_A^{\AF}$ such that 
\begin{equation*}
g_1(x_0) \ne g_1(\tau_0(x_0)), \qquad 
\sigma_A^{K_0}(x) = \sigma_A^{K_0}(\tau_0(x_0)).
\end{equation*}
Since $h:X_A\longrightarrow X_B$ gives rise to a continuous orbit equivalence
between
$(X_A,\sigma_A)$ and $(X_B,\sigma_B)$,
we have by \cite[Lemma 5.1]{MaPacific},
\begin{equation*}
\sigma_B^{k_1^{K_0}(x_0)}(h(\sigma_A^{K_0}(x_0))) =
\sigma_B^{l_1^{K_0}(x_0)}(h(x_0)) 
\end{equation*}
and hence
\begin{equation*}
\sigma_B^{k_1^{K_0}(x_0)}(h(\sigma_A^{K_0}(\tau_0(x_0)))) =
\sigma_B^{l_1^{K_0}(x_0)}(h(x_0)), 
\end{equation*}
so that we have
\begin{equation*}
\sigma_B^{k_1^{K_0}(x_0) + k_1^{K_0}(\tau_0(x_0)) }(h(\sigma_A^{K_0}(\tau_0(x_0)))) =
\sigma_B^{l_1^{K_0}(x_0) + k_1^{K_0}(\tau_0(x_0))}(h(x_0)).  
\end{equation*}
Hence we have
\begin{equation}
\sigma_B^{k_1^{K_0}(x_0) + l_1^{K_0}(\tau_0(x_0)) }(h(\tau_0(x_0))) =
\sigma_B^{l_1^{K_0}(x_0) + k_1^{K_0}(\tau_0(x_0))}(h(x_0)).  \label{eq:3.5.2}
\end{equation}
By the hypothesis with the condition $\tau_0 \in \Gamma_A^{\AF}$,
we see that 
$\xi_h(\tau_0)\in \Gamma_B^{b_1\circ h^{-1}}
$
so that  
\begin{equation}
\sigma_B^{k_{\xi_h(\tau_0)}(h(x_0)) }(h(\tau_0(x_0))) =
\sigma_B^{l_{\xi_h(\tau_0)}(h(x_0))}(h(x_0)) \label{eq:3.5.3}
\end{equation}
and
\begin{align}
l_{\xi_h(\tau_0)}(h(x_0)) -k_{\xi_h(\tau_0)}(h(x_0))
= & b_1 \circ h^{-1}(h(x_0)) -  b_1 \circ h^{-1}(\xi_h(\tau_0)(h(x_0)) )\\
= & b_1(x_0) -  b_1 (\tau_0(x_0)). \label{eq:3.5.4}
\end{align}
Since 
$d_{\xi_h(\tau_0)}(h(x_0)) 
=l_{\xi_h(\tau_0)}(h(x_0)) -k_{\xi_h(\tau_0)}(h(x_0))$,
the equalities
\eqref{eq:3.5.2},
\eqref{eq:3.5.3},
\eqref{eq:3.5.4}
imply 
\begin{equation*}
(l_1^{K_0}(x_0) + k_1^{K_0}(\tau_0(x_0)))
-
(k_1^{K_0}(x_0) + l_1^{K_0}(\tau_0(x_0)))
= b_1(x_0) -  b_1 (\tau_0(x_0))
\end{equation*}
so that 
\begin{equation}
\{ l_1^{K_0}(x_0) - k_1^{K_0}(x_0) \}
-
\{ l_1^{K_0}(\tau_0(x_0)) - k_1^{K_0}(\tau_0(x_0)) \}
= b_1(x_0) -  b_1 (\tau_0(x_0)). \label{eq:3.5.5}
\end{equation}
Hence we have
$$
c_1^{K_0}(x_0) -b_1(x_0) = c_1^{K_0}(\tau_0(x_0)) -b_1(\tau_0(x_0)) 
$$
and
\begin{equation}
c_1^{K_0}(x_0) -b_1(x_0) + b_1(\sigma_A^{K_0}(x_0))
= c_1^{K_0}(\tau_0(x_0)) -b_1(\tau_0(x_0)) 
+b_1(\sigma_A^{K_0}(\tau_0(x_0))). \label{eq:3.5.6}
\end{equation}
We then have
\begin{align*}
\text{LHS of } \eqref{eq:3.5.6}
=& \sum_{i=0}^{K_0-1}c_1(\sigma_A^i(x_0))
 -\{ \sum_{i=0}^{K_0-1}b_1(\sigma_A^i(x_0))
 + \sum_{i=0}^{K_0-1}b_1(\sigma_A^i(\sigma_A(x_0))\} \\
= & \sum_{i=0}^{K_0-1}g_1(\sigma_A^i(x_0)) = g_1^{K_0}(x_0)\\
\intertext{ and similarly}
\text{RHS of } \eqref{eq:3.5.6}
= & \sum_{i=0}^{K_0-1}g_1(\sigma_A^i(\tau_0(x_0))) 
= g_1^{K_0}(\tau_0(x_0)).\\
\end{align*}
As $g_1(x_0) \ne g_1(\tau_0(x_0))$, 
there exists $m_0 \in \N$ with $1\le m_0\le K_0-1$ such that 
\begin{equation}
g_1(\sigma_A^{m_0}(x_0)) \ne g_1(\sigma_A^{m_0}(\tau_0(x_0))).
\label{eq:3.5.9}
\end{equation}
Put $x_1 =\sigma_A^{m_0}(x_0)$.
As $\sigma_A^{K_0 -m_0}(x_1) =\sigma_A^{K_0 -m_0}(\sigma_A^{m_0}(\tau_0(x_0)))$,
one may find $\tau_1 \in \Gamma_A^{\AF}$ such that 
\begin{equation}
\tau_1(x_1) =\sigma_A^{m_0}(\tau_0(x_0)). 
\label{eq:3.5.10}
\end{equation}
By \eqref{eq:3.5.9} and  \eqref{eq:3.5.10},
we have
\begin{equation*}
g_1(x_1) \ne g_1(\tau_1(x_1)) 
\quad
\text{ and }
\quad
\sigma_A^{K_0 -m_0}(x_1) =\sigma_A^{K_0 -m_0}(\tau_1(x_1)).
\end{equation*}
This implies $K_0 -m_0 \in \mathbb{K}$
a contradiction to the minimality of $K_0 =\min{\mathbb{K}}$.
\end{proof}

\begin{proposition}\label{prop:3.4}
Assume that the matrices $A$ and $B$ are both primitive.
Suppose that there exist a homeomorphism $h:X_A\longrightarrow X_B$
that gives rise to
$(X_A,\sigma_A) \underset{\coe}{\sim}(X_B,\sigma_B)$
and  continuous functions $b_1 \in C(X_A,\Z), b_2 \in C(X_B,\Z)$ such that 
\begin{equation*}
\xi_h (\Gamma_A^{\AF}) = \Gamma_B^{b_1\circ h^{-1}}, \qquad
\xi_{h^{-1}} (\Gamma_B^{\AF}) = \Gamma_A^{b_2\circ h}.
\end{equation*}
Then the equalities 
\begin{equation}\label{eq:3.6.01}
c_1 = 1 + b_1 - b_1\circ\sigma_A, \qquad 
c_2 = 1 + b_2 - b_2\circ\sigma_B
\end{equation}
hold, so that 
$(X_A,\sigma_A) \underset{\scoe}{\sim}(X_B,\sigma_B)$.
\end{proposition}
\begin{proof}
By Lemma \ref{lem:3.5}, the function 
$c_1 - (b_1 - b_1\circ\sigma_A)$
and the functin  
$c_2 - (b_2 - b_2\circ\sigma_B)$
are both constants.
They are denoted by $N_1$ and $ N_2,$
respectively. 
Hence we have  for $y = h(x) \in X_B$
\begin{align*}
N_1 \cdot l_2(y) 
=& \sum_{i=0}^{l_2(y) -1} c_1(\sigma_A^i(x)) 
- \sum_{i=0}^{l_2(y) -1} b_1(\sigma_A^i(x))  
+ \sum_{i=0}^{l_2(y) -1}b_1(\sigma_A^i(\sigma_A(x))) \\
= & c_1^{l_2(y)}(x) - b_1(x) + b_1(\sigma_A^{l_2(y)}(x))\\
\intertext{and similarly} 
N_1 \cdot k_2(y) 
=&   c_1^{k_2(y)}(h^{-1}(\sigma_B(y))) - b_1(h^{-1}(\sigma_B(y))) 
+ b_1(\sigma_A^{k_2(y)}(h^{-1}(\sigma_B(y)))),
\end{align*}
so that 
\begin{align}
c_1^{l_2(y)}(x) 
=&  N_1 \cdot l_2(y)  + b_1(x) - b_1(\sigma_A^{l_2(y)}(x)), \label{eq:3.6.1}\\
c_1^{k_2(y)}(h^{-1}(\sigma_B(y)))
= & N_1 \cdot k_2(y) +b_1(h^{-1}(\sigma_B(y))) -b_1(\sigma_A^{k_2(y)}(h^{-1}(\sigma_B(y)))).
\label{eq:3.6.2}
\end{align}
Since 
$(X_A,\sigma_A) \underset{\coe}{\sim}(X_B,\sigma_B)$
via homeomorphism $h:X_A\longrightarrow X_B$,
we have
\begin{equation*}
\varPsi_h(1_B) = c_1, \qquad
\varPsi_{h^{-1}}(1_A) = c_2
\end{equation*}
and hence  
$\varPsi_h(c_2) = 1_A, \, 
\varPsi_{h^{-1}}(c_1) = 1_B$.
Since
$$\varPsi_h(c_2)(x) = c_2^{l_1(x)}(h(x)) -  c_2^{k_1(x)}(h(\sigma_A(x))),
$$ 
we have 
\begin{equation}
c_2^{l_1(x)}(h(x)) -  c_2^{k_1(x)}(h(\sigma_A(x))) =1, \label{eq:3.6.3}
\end{equation}
and similarly
\begin{equation}
c_1^{l_2(y)}(h^{-1}(y)) -  c_1^{k_2(y)}(h^{-1}(\sigma_B(y))) =1. \label{eq:3.6.4}
\end{equation}
For $y = h(x)$, by \eqref{eq:3.6.1},\eqref{eq:3.6.2} and \eqref{eq:3.6.4},
we have
\begin{align*}
& \{N_1 \cdot l_2(y)  + b_1(x) - b_1(\sigma_A^{l_2(y)}(x))\} \\
- & \{ N_1 \cdot k_2(y) +b_1(h^{-1}(\sigma_B(y))) -b_1(\sigma_A^{k_2(y)}(h^{-1}(\sigma_B(y))))\} = 1
\end{align*}
so that 
\begin{equation*}
N_1 ( l_2(y) -  k_2(y)) + b_1(x) - b_1(h^{-1}(\sigma_B(y))) =1
\end{equation*}
and hence  
\begin{equation*}
N_1 \cdot c_2(y)  + b_1(h^{-1}(y)) - b_1(h^{-1}(\sigma_B(y))) =1, \qquad y \in X_B.
\end{equation*}
As $c_2(y) = N_2 + b_2(y) - b_2(\sigma_B(y))$,
we have
\begin{equation*}
N_1 (N_2 + b_2(y) - b_2(\sigma_B(y))  + b_1(h^{-1}(y)) - b_1(h^{-1}(\sigma_B(y))) =1, \qquad y \in X_B
\end{equation*}
so that 
 \begin{equation*}
N_1 N_2  -1 + N_1 b_2(y)   + b_1(h^{-1}(y)) 
= N_1 b_2(\sigma_B(y)) + b_1(h^{-1}(\sigma_B(y))) , \qquad y \in X_B.
\end{equation*}
We put $g(y) = N_1 b_2(y)   + b_1(h^{-1}(y)), y \in X_B$
and hence we get
\begin{equation}
N_1 N_2 -1 + g(y) = g(\sigma_B(y)), \qquad   y \in X_B
\label{eq:3.6.7}
\end{equation}
By \eqref{eq:3.6.7}, we have
\begin{equation}
(N_1 N_2 -1)\cdot n  =  g(\sigma_B^n(y)) -g(y), \qquad   y \in X_B, \, n \in \N.
\label{eq:3.6.8}
\end{equation}
Since the function $y \in X_B \longrightarrow g(y) \in \Z$ 
is continuous and hence bounded, 
the equalities \eqref{eq:3.6.8} do not hold unless $N_1 N_2 =1$.
The equality
$c_1(x) =N_1 - \{ b_1(x) - b_1(\sigma_A(x))\}, x \in X_A$
implies
$[c_1] =[N_1] $ in the ordered cohomology 
$H^A = C(X_A,\Z)/\{ f - f\circ\sigma_A\mid f \in C(X_A,\Z)\}$.
By a key ingredient of the paper \cite[Lemma 5.4]{MMETDS},
we know that $[c_1] = [\varPsi_h(1_B)]$ 
belongs to the positive cone $H^A_+$ of $H^A$.  
Hence we conclude that $N_1 =1$ and hence $N_2=1$.
We thus conclude that the equalities
\eqref{eq:3.6.01}  hold.
\end{proof}

\medskip

Now we will give a proof of Theorem \ref{thm:main2}.

\medskip

{\it Proof of Theorem \ref{thm:main2}:}

\noindent
The equivalence between (i) and (ii) was proved  in \cite[Theorem 6.7]{MaJOT}.

\noindent
(i) $\Longrightarrow$ (iv):
Assume that 
$(X_A,\sigma_A)$ and $(X_B,\sigma_B)$ are strongly continuous orbit equivalent.
Proposition \ref{prop:3.2} tells us that the assertion (iv) holds.

\noindent
(iv) $\Longrightarrow$ (iii):
Under the statement (iv), take the functions $f\equiv 0, g\equiv 0$, then
the statement (iii) follows.

\noindent
Assume the assertion (iii). 
Proposition \ref{prop:3.4} tells us that the assertion (i) holds.

\noindent
The equivalence between (i) and (v) is proved in Proposition \ref{prop:scoecoho}.

Therefore we completed the proof of Theorem \ref{thm:main2}.

\section{$\Gamma$-strongly continuous orbit equivalence}
In this final section, we will study an  equivalence relation 
slightly stronger than strongly continuous orbit equivalence.  
Let $h:X_A\longrightarrow X_B$ 
be a homeomorphism that gives rise to a continuous orbit equivalence
between $(X_A,\sigma_A)$ and $(X_B,\sigma_B)$ as in 
\eqref{eq:orbiteq1x} and \eqref{eq:orbiteq2y}.
The cocycle functions are defined by 
$ c_1 = l_1 - k_1$ and $c_2 = l_2 - k_2$.
Let us denote by $1_A, 1_B$  the constant functions whose values are everywhere $1$ on $X_A, X_B$,
  respectively.
Recall that one-sided topological Markov shifts 
$(X_A,\sigma_A)$ and $(X_B,\sigma_B)$ are strongly continuous orbit equivalent
if there exists a continuous function $b_1:X_A\longrightarrow \Z$ such that
$c_1 = 1_A + b_1 - b_1\circ\sigma_A$.
In this case we may find a continuous function $b_2:X_B\longrightarrow \Z$ such that
$c_2 = 1_B + b_2 - b_2\circ\sigma_B$ (\cite[Lemma 4.3]{MaJOT}). 
\begin{definition}
One-sided topological Markov shifts $(X_A,\sigma_A)$ and $(X_B,\sigma_B)$ 
are said to be $\Gamma$-{\it strongly continuous orbit equivalent}\/ written
$ (X_A,\sigma_A)\underset{\Gamma-scoe}{\sim}(X_B,\sigma_B)$
if there exists a homeomorphism $h:X_A\longrightarrow X_B$ that gives rise to a 
continuous orbit equivalence between $(X_A,\sigma_A)$ and $(X_B,\sigma_B)$ and there exists 
an element $\tau \in \Gamma_A$ such that 
$c_1 = 1_A - d_\tau + d_\tau\circ\sigma_A$.
\end{definition}
The above definition means that by putting $b_1 = -d_\tau$, 
the equality $c_1 = 1_A + b_1 - b_1 \circ \sigma_A$
holds, so that $\Gamma$-strongly continuous orbit equivalence 
implies strongly continuous orbit equivalence.
Recall that $\xi_h(\tau) \in \Gamma_B$
is defined by $\xi_h(\tau) = h\circ \tau\circ h^{-1}$ for $\tau \in \Gamma_A$. 
\begin{lemma} Keep the above notation.
Suppose that $c_1 = 1_A - d_\tau + d_\tau \circ \sigma_A$
for some $\tau \in \Gamma_A$.
Then the equality
$c_2 = 1_B - d_{\xi_h(\tau^{-1})} + d_{\xi_h(\tau^{-1})}\circ \sigma_B$ 
holds.
\end{lemma}
\begin{proof}
We first prove that the following identities 
\begin{equation} \label{eq: step 25}
d_\tau \circ h^{-1} = d_{\xi_h(\tau)}\circ \xi_h(\tau^{-1})= - d_{\xi_h(\tau^{-1})}
\end{equation}
hold.
By \cite[Lemma 5.2]{MaJOT}, we have
\begin{align*}
  & c_1^{l_\tau(x)}(x) - c_1^{k_\tau(x)}(\tau(x)) \\
= & \{ l_\tau(x) -d_\tau(x) + d_\tau(\sigma_A^{l_\tau(x)}(x))\}
-  \{ k_\tau(x) -d_\tau(\tau(x)) +d_\tau(\sigma_A^{k_\tau(x)}(\tau(x)))\} \\
= & d_\tau(\tau(x)). 
\end{align*}
By \eqref{eq:3.1}, we obtain 
\begin{equation} \label{eq:step23}
d_\tau(\tau(x)) = d_{\xi_h(\tau)}(h(x)), \qquad x \in X_A.
\end{equation}
By putting $y = h(\tau(x))$ in \eqref{eq:step23},
we have
\begin{equation*} \label{eq:step33}
d_\tau(h^{-1}(y)) 
= d_{\xi_h(\tau)}(h\circ \tau^{-1}\circ h^{-1}(y)))
= d_{\xi_h(\tau)}(\xi_h(\tau^{-1})(y)), \qquad y \in X_B.
\end{equation*}
We will next show the identity
\begin{equation*} \label{eq:step28}
d_{\xi_h(\tau)}\circ \xi_h(\tau^{-1}) = - d_{\xi_h(\tau^{-1})}.
\end{equation*}
Since $\xi_h(\tau)^{-1} = \xi_h(\tau^{-1})$, we have 
by using \cite[Lemma 7.7]{MMGGD}, 
\begin{equation*}
d_{\xi_h(\tau)}\circ \xi_h(\tau^{-1}) 
=d_{\xi_h(\tau)}\circ \xi_h(\tau)^{-1} 
=- d_{{\xi_h(\tau)}^{-1}} = - d_{\xi_h(\tau^{-1})}. 
\end{equation*}
Hence, the identities \eqref{eq: step 25} hold.

Since 
$c_1 = \Psi_h(1_B),c_2 = \Psi_{h^{-1}}(1_A)$ and
$\Psi_{h^{-1}} = \Psi_h^{-1}$
by \cite[Corolary 3.8]{MaJOT}, 
we have 
$$
c_2 
=  \Psi_{h^{-1}}(1_A) 
= \Psi_{h^{-1}}(c_1 +d_\tau - d_\tau\circ \sigma_A)
= \Psi_{h^{-1}}(c_1)+\Psi_{h^{-1}}(d_\tau - d_\tau\circ \sigma_A).
$$
By \cite[Lemma 4.6]{MMETDS},  the equality
$
\Psi_{h^{-1}}(d_\tau - d_\tau\circ \sigma_A) 
= d_\tau\circ h^{-1} - d_\tau \circ h^{-1} \circ \sigma_B$
holds so that
$$
c_2 
= 1_B + d_\tau \circ h^{-1} - d_\tau\circ h^{-1}\circ \sigma_B.
$$
We thus have the desired equality by 
\eqref{eq: step 25}.
\end{proof}

Therefore we have the following proposition.
\begin{proposition}
The following are equivalent.
\begin{enumerate}
\renewcommand{\theenumi}{\roman{enumi}}
\renewcommand{\labelenumi}{\textup{(\theenumi)}}
\item $ (X_A,\sigma_A)\underset{\Gamma-scoe}{\sim}(X_B,\sigma_B)$.
\item 
There exists a homeomorphism $h:X_A\longrightarrow X_B$ that gives rise to a 
continuous orbit equivalence between $(X_A,\sigma_A)$ and $(X_B,\sigma_B)$
with cocycle functions $c_1 = l_1 - k_1 \in C(X_A, \Z)$
and $c_2 = l_2 - k_2 \in C(X_B,\Z)$ such that  
 there exist elements of their continuous  full groups 
 $\tau_1 \in \Gamma_A, \tau_2 \in \Gamma_B$
   such that $\tau_2 = h \circ \tau_1^{-1}\circ h^{-1}$
   and
\begin{equation*}
c_1 = 1_A - d_{\tau_1} +  d_{\tau_1} \circ\sigma_A, \qquad 
c_2 = 1_B - d_{\tau_2} +  d_{\tau_2} \circ\sigma_B.
\end{equation*}
\end{enumerate}
\end{proposition}

\begin{lemma}\label{lem:6.4}
Keep the notation.
If the cocycle function $c_1 \equiv 1_A$, then the other 
 cocycle function $c_2 \equiv 1_B$.
 Hence  in this case, the one-sided topological Markov shifts 
$(X_A,\sigma_A)$ and $(X_B,\sigma_B)$ become eventually conjugate.
\end{lemma}
\begin{proof}
Since $c_1 = \Psi_h(1_B)$ and $ c_2 = \Psi_{h^{-1}}(1_A),$
the condition $c_1 \equiv 1_A$ implies 
$c_2 \equiv 1_B$, because $\Psi_{h^{-1}} = \Psi_h^{-1}$
by \cite[Corollary 3.8]{MaJOT}.
\end{proof}

We will next prove the following proposition.
\begin{proposition}
$\Gamma$-strongly continuous orbit equivalence implies eventual conjugacy.
\end{proposition}
\begin{proof}
Suppose that there exists a homeomorphism $h:X_A\longrightarrow X_B$
that gives rise to $(X_A,\sigma_A) \underset{\Gamma-scoe}{\sim}(X_B, \sigma_B)$
such that 
\begin{equation*}
c_1 = 1_A -d_{\tau_1} + d_{\tau_1} \circ \sigma_A
\quad
\text{ and }
\quad
c_2 = 1_B -d_{\tau_2} + d_{\tau_2} \circ \sigma_B
\end{equation*}
 for some 
$\tau_1 \in \Gamma_A$ and 
$\tau_2 \in \Gamma_B$ satisfying $\tau_2 = h\circ \tau_1^{-1}\circ h^{-1}.$
As in \eqref{eq: step 25},
we have
\begin{equation}\label{eq:step31}
d_{\tau_2}\circ h = - d_{\tau_1}.
\end{equation}
Let $k_1:X_A\longrightarrow \Z$ be the continuous function appeard in \eqref{eq:orbiteq1x}.
Take $K \in \N$ such as 
$K =\Max\{k_1(x), -d_{\tau_2}(h(x))\mid x \in X_A\}.$
By using \eqref{eq:step31}, we then have 
\begin{align*}
\sigma_B^{K} (h(\sigma_A(x))) 
& = \sigma_B^{K + c_1(x)}(h(x)) \\
& = \sigma_B^{K + 1 - d_{\tau_1}(x) +  d_{\tau_1}(\sigma_A(x))}(h(x)) \\
& = \sigma_B^{K + 1 + d_{\tau_2}(h(x)) - d_{\tau_2}(h(\sigma_A(x)))}(h(x))
\end{align*}
so that 
\begin{equation*}
\sigma_B^{K + d_{\tau_2}(h(\sigma_A(x)))} (h(\sigma_A(x))) 
 = \sigma_B^{K + 1 + d_{\tau_2}(h(x))}(h(x)).
\end{equation*}
Since we have 
\begin{equation*}
\sigma_B^{l_{\tau_2}(h(y))} (h(y)) 
 = \sigma_B^{k_{\tau_2}(h(y))} (\tau_2(y))
\quad 
\text{ for all } y \in X_B,
\end{equation*}
we have
\begin{equation*}
\sigma_B^{K } (\tau_2(h(\sigma_A(x)))) 
 = \sigma_B^{K + 1}(\tau_2(h(x))).
\end{equation*}
By putting $h' = \tau_2\circ h: X_A\longrightarrow X_B$,
we obtain that 
\begin{equation*}
\sigma_B^{K } (h'(\sigma_A(x)))) 
 = \sigma_B^{K + 1}(h'h(x))), \qquad x \in X_A,
\end{equation*}
proving that 
$(X_A,\sigma_A)$ and $(X_B,\sigma_B)$ are eventually conjugate. 
\end{proof}
We thsu have the following theorem.
\begin{theorem}\label{thm:main6.6}
Let $A, B$ be irreducible, non-permutation matrices with entries in $\{0,1\}$.
Then the one-sided topological Markov shifts 
$(X_A, \sigma_A)$ and $(X_B, \sigma_B)$ are $\Gamma$-strongly continuous orbit equivalent
if and only if they are eventually conjugate. 
\end{theorem}
As a corollary we know the following result.
\begin{corollary}\label{cor:main6.7}
Let $A, B$ be irreducible, non-permutation matrices with entries in $\{0,1\}$.
Then the following three conditions are equivalent:
\begin{enumerate}
\renewcommand{\theenumi}{\roman{enumi}}
\renewcommand{\labelenumi}{\textup{(\theenumi)}}
\item $(X_A,\sigma_A)$ and $(X_B,\sigma_B)$ are uniformly continuous orbit equivalent. 
\item $(X_A,\sigma_A)$ and $(X_B,\sigma_B)$ are eventually conjugate. 
\item $(X_A, \sigma_A)$ and $(X_B, \sigma_B)$ are $\Gamma$-strongly continuous orbit equivalent.
\end{enumerate}
\end{corollary}
\begin{proof}
The equivalence (i) $\Longleftrightarrow$ (ii) was proved in \cite{MaPAMS2017}.
The equivalence (ii) $\Longleftrightarrow$ (iii) is shown in Theorem \ref{thm:main6.6}.
\end{proof}



\bigskip

{\it Acknowledgment:}
This work was supported by JSPS KAKENHI 
Grant Numbers 15K04896, 19K03537.

\end{document}